\documentclass[11pt]{amsart}
\usepackage{amssymb,amscd,amsthm,latexsym}
\allowdisplaybreaks[3]
\usepackage{graphicx}
\usepackage{hyperref}
\usepackage[british]{babel}
\usepackage[all]{xy}
\usepackage{color}
\usepackage{mathabx}
\usepackage{calrsfs}
\usepackage{booktabs}

\theoremstyle{plain}
\newtheorem{theorem}[subsection]{Theorem}
\newtheorem{lemma}[subsection]{Lemma}
\newtheorem{proposition}[subsection]{Proposition}

\theoremstyle{definition}
\newtheorem{definition}[subsection]{Definition}

\theoremstyle{remark}
\newtheorem{example}[subsection]{Example}

\newtheorem{remark}[subsection]{Remark}

\newdir{ >}{
 @{}*!/-10pt/@{>} }
\newdir{ |>}{
 @{}*!/-5.5pt/@{|}*!/-10pt/:(1,-.2)@^{>}*!/-10pt/:(1,+.2)@_{>} }
\newdir{ >>}{{}*!/8pt/@{|}*!/3.5pt/:(1,-.2)@^{>}*!/3.5pt/:(1,+.2)@_{>}}

\newcommand{\defn}[1]{\emph{#1}}

\def\mathrmdef#1{\expandafter\def\csname#1\endcsname{{\rm#1}}}
\mathrmdef{max}
\mathrmdef{op}
\mathrmdef{id}

\newcommand{\C}{\ensuremath{\mathbb{C}}}
\newcommand{\Ccal}{\ensuremath{\mathsf{C}}}

\newcommand{\V}{\ensuremath{\mathcal{V}}}
\newcommand{\Grp}{\ensuremath{\mathsf{Grp}}}
\newcommand{\Ord}{\ensuremath{\mathsf{Ord}}}
\newcommand{\Pos}{\ensuremath{\mathsf{Pos}}}

\newcommand{\Mon}{\ensuremath{\mathsf{Mon}}}
\newcommand{\MMon}{\ensuremath{\mathbb{M}\mathsf{on}}}
\newcommand{\GMon}{\ensuremath{\mathsf{GMon}}}
\newcommand{\GGMon}{\ensuremath{\mathbb{G}\mathsf{Mon}}}

\newcommand{\Set}{\ensuremath{\mathsf{Set}}}
\newcommand{\N}{\ensuremath{\mathbb{N}}}

\newcommand{\OrdGrp}{\ensuremath{\mathsf{OrdGrp}}}
\newcommand{\OOrdGrp}{\ensuremath{\mathbb{O}}\ensuremath{\mathsf{rdGrp}}}

\newcommand{\VCat}{\ensuremath{V\mbox{-}\mathsf{Cat}}}
\renewcommand{\O}{\ensuremath{\mathsf{O}}}
\newcommand{\OO}{\ensuremath{\mathbb{O}}}

\def\mathrmdef#1{\expandafter\def\csname#1\endcsname{{\rm#1}}}
\newcommand{\la}{\langle}
\newcommand{\ra}{\rangle}
\newcommand{\mono}{\rightarrowtail}
\newcommand{\repi}{\ensuremath{\twoheadrightarrow}}
\newcommand{\moi}{\preccurlyeq}
\newcommand{\co}{\mathrm{co}}
\newcommand{\nt}{\circ}

\def\vpullback{
\ar@{-}[]+D+<6pt,-6pt>;[]+D+<0pt,-12pt>;%
\ar@{-}[]+D+<0pt,-12pt>;[]+D+<-6pt,-6pt>}

\def\pullback{
 \ar@{-}[]+R+<4pt,-3pt>;[]+RD+<4pt,-6pt>%
 \ar@{-}[]+D+<3pt,-6pt>;[]+RD+<4pt,-6pt>}

\usepackage{chngcntr}
\counterwithin*{equation}{section}

\def\mathrmdef#1{\expandafter\def\csname#1\endcsname{{\rm#1}}}
\mathrmdef{max}
\mathrmdef{op}
\mathrmdef{sym}

\newcommand{\Rel}{\ensuremath{\mathsf{Rel}}}
\newcommand{\Relf}{\ensuremath{\mathsf{Rel_{ff}}}}
\newcommand{\Relw}{\ensuremath{\mathsf{Rel_{idl}}}}

\def\relto{{\rightarrow\hspace*{-2.4ex}{\mapstochar}\hspace*{2.2ex}}}

\begin{document}

\title{Enriched aspects of calculus of relations and $2$-permutability}

\author{Maria Manuel Clementino}
\address{CMUC, Department of Mathematics, University of
Coimbra, 3000-143 Coimbra, Portugal}\thanks{}
\email{mmc@mat.uc.pt}

\author{Diana Rodelo}
\address{Department de Mathematics, University of the Algarve, 8005-139 Faro, Portugal and CMUC, Department of Mathematics, University of Coimbra, 3000-143 Coimbra, Portugal}
\thanks{The authors acknowledge partial financial support by {\it Centro de Matemática da Universidade de Coimbra} (CMUC), funded by the Portuguese Government through FCT/MCTES, DOI 10.54499/UIDB/00324/2020.}
\email{drodelo@ualg.pt}

\keywords{regular categories, categories enriched in the category of preorders, Mal'tsev categories, Mal'tsev objects, $V$-categories}

\subjclass{18E08, 
18E13, 
18D20, 
08C05, 
18N10, 
}

\begin{abstract}
The aim of this work is to further develop the calculus of (internal) relations for a regular $\Ord$-category $\C$. To capture the enriched features of a regular $\Ord$-category and obtain a good calculus, the relations we work with are precisely the \emph{ideals} in $\C$. We then focus on an enriched version of the 1-dimensional algebraic $2$-permutable (also called Mal'tsev) property and its well-known equivalent characterisations expressed through properties on ordinary relations.  We introduce the notion of \emph{$\Ord$-Mal'tsev category} and show that these may be characterised through enriched versions of the above mentioned properties adapted to ideals. Any $\Ord$-enrichment of a 1-dimensional Mal'tsev category is necessarily an $\Ord$-Mal'tsev category. We also give some examples of categories which are not Mal'tsev categories, but are $\Ord$-Mal'tsev categories.
\end{abstract}

\maketitle

\section*{Introduction}
The notion of \emph{regular category}~\cite{Barr} has been widely studied and explored in Category Theory over the past 50 years. Regular categories capture several nice exactness properties of abelian categories~\cite{Buchsbaum}, one of the notions in the genesis of Category Theory, but without requiring them to be additive. A handy exactness property of regular categories is the existence of (regular) images. This makes regular categories a good context to work with ordinary relations, since it is possible to define their composition and this composition is associative. The calculus of ordinary relations provides a well established and powerful tool for obtaining proofs in regular categories. Another good reason for the successful development of regular categories is the large number of examples. The category of sets, any elementary topos, abelian categories or any variety of universal algebras are all examples of Barr-exact categories~\cite{Barr}, which are regular categories. The category of topological groups gives an example of a regular category which is not Barr-exact (see~\cite{C21}).

A variety of universal algebras, of a certain type, is defined through its signature and axioms, i.e. its theory admits specific operations satisfying given identities. For example, a variety of universal algebras is called a \emph{$2$-permutable variety}~\cite{Smith} (they are also called \emph{congruence permutable varieties} or \emph{Mal'tsev varieties}) when its theory admits a ternary Mal'tsev operation $p$ satisfying the identities $p(x,y,y)=x$ and $p(x,x,y)=y$. The variety $\Grp$ of groups is $2$-permutable, where $p(x,y,z)=xy^{-1}z$. Sometimes it is possible to extract from the operations and identities equivalent properties involving homomorphic relations. The translation of these properties on homomorphic relations to an appropriate categorical setting could be used to define the categorical counterpart of such type of variety. For example, the existence of a Mal'tsev operation of a $2$-permutable variety $\V$ is equivalent to the fact that the composition of any pair of congruences $R,S$ on any algebra $X$ of $\V$ is $2$-permutable (=commutative): $RS\cong SR$~\cite{Maltsev-Sbornik}. It was shown in~\cite{Lambek} that $2$-permutable varieties can also be characterised by the fact that any homomorphic relation $D$ from an algebra $X$ to an algebra $Y$ is \emph{difunctional}:
\[
	\left[\,(x,y)\in D \wedge (u, y)\in D \wedge (u,v)\in D\,\right] \Rightarrow (x,v)\in D,
\]
where $x,u\in X$ and $y,v\in Y$. The notion of $2$-permutable variety was generalised to a categorical context in~\cite{CLP} (see also~\cite{CKP, CPP, BB}). This was achieved by translating the characteristic properties on homomorphic relations of $2$-permutable varieties into similar properties on ordinary relations for categories. A regular category $\Ccal$ is called a \emph{Mal'tsev category} when any pair of ordinary equivalence relations $R, S$  on any object $X$ in $\Ccal$ is such that $RS\cong SR$. Without requiring any kind of exactness properties, a category $\Ccal$ is a Mal'tsev category when any ordinary relation $D\colon X \relto Y$ in $\Ccal$ is difunctional, i.e. the relation $\Ccal(W,D)\colon  \Ccal(W,X)\relto \Ccal(W,Y)$ in $\Set$ is difunctional, for every object $W$ of $\Ccal$ (see Definition~\ref{def of difunctional}). There are several alternative well-known characterisations of regular Mal'tsev categories given through other properties on ordinary relations, such as: every ordinary reflexive relation is an ordinary equivalence relation. They are recalled in Theorem~\ref{from CKP}.

The aim of this work is to explore $2$-permutability in an $\Ord$-enriched context and define, what we call, $\Ord$-Mal'tsev category -- Section~\ref{Mal'tsev property in the Ord-enriched context}. To do so we consider an enriched version of the property concerning the difunctionality of ordinary relations. The appropriate enriched version of an ordinary relation turns out to be that of \emph{ideal} (see Definition~\ref{w-c relation}). An $\Ord$-category $\C$ is called an \emph{$\Ord$-Mal'tsev category} when every ideal $D\colon X\looparrowright Y$ in $\C$ satisfies the property: given morphisms $x, u, u'\colon A\to X$, $y,y',v\colon A\to Y$, the following implication holds
\[
	\left[\, (x,y)\in_A D \wedge y\moi y' \wedge (u,y')\in_A D \wedge u\moi u' \wedge (u',v)\in_A D\,\right] \Rightarrow (x,v)\in_A D.
\]
Any $\Ord$-enrichment of a Mal'tsev category is necessarily an $\Ord$-Mal'tsev category. If an $\Ord$-category $\C$ is regular (see Definition~\ref{enriched regular cat} below) then we obtain equivalent characterisations of regular $\Ord$-Mal'tsev categories through properties on ideals (Theorem~\ref{Ord-enriched for CKP}, which is the enriched version of Theorem~\ref{from CKP}).

A fundamental part of this work concerns the characterisations of regular $\Ord$-Mal'tsev categories obtained in Theorem~\ref{Ord-enriched for CKP}. This is achieved by developing an enriched calculus of relations for ideals in the context of regular $\Ord$-categories -- Section~\ref{Order ideals and their calculus of relations}. We adapt the calculus of relations given in~\cite{Vassilis}, which was done for regular $\Pos$-categories, to regular $\Ord$-categories and further explore the possible extensions of the known calculus of ordinary relations in the regular context (see~\cite{CKP}).

We give examples of categories which are not Mal'tsev categories, and provide them with an $\Ord$-enrichment for which they are $\Ord$-Mal'tsev categories -- Section~\ref{Examples}. The example concerning the category $(\VCat)^\op$ relies on an object-wise approach to $\Ord$-Mal'tsev categories, which is developed in the Appendix.

\section{\texorpdfstring{$\Ord$}{Ord}-enriched categories}\label{Ord-enriched categories}

Let $\C$ be an $\Ord$-category, i.e. a category enriched in the category $\Ord$ of preordered sets (i.e. sets equipped with a reflexive and transitive relation) and monotone maps. This means that, for any objects $X$ and $Y$ of $\C$, $\C(X,Y)$ is equipped with a preorder such that (pre)composition preserves it. We will denote this preorder of morphisms by $\preccurlyeq$. If we consider in $\C$ the reverse preorder we obtain again an $\Ord$-enriched category which we denote, as usual, by $\C^\co$. Any category $\Ccal$ with the identity order (=discrete order) on morphisms can be considered an $\Ord$-category.

A morphism $m\colon X\to Y$ is said to be \defn{full} when: given morphisms $a,a'\colon A\to X$ such that $ma \moi ma'$, then $a\moi a'$; equivalently, $ma\moi ma'$ if and only if $a\moi a'$. (Note that, in the $\Ord$-enriched context, all morphisms are faithful.) Such (mono)morphisms are also called \defn{ff-(mono)morphisms}, where the ``ff'' stands for ``fully faithful''; see~\cite{KV, Vassilis}. If the preorder $\preccurlyeq$ is also antisymmetric, so that $\C$ is a $\Pos$-category, then an ff-morphism is necessarily a monomorphism; this is not the case when $\C$ is an $\Ord$-category. We denote ff-monomorphisms with arrows of the type $\mono$. We have similar properties for ff-(mono)morphisms as those of monomorphisms in ordinary categories.

\begin{lemma}\label{pps for ff-ms}
Let $m\colon X\to Y$ and $n\colon Y\to Z$ be morphisms in an $\Ord$-category $\C$. Then:
\begin{itemize}
	\item[(1)] if $m$ and $n$ are ff-(mono)morphisms, then $nm$ is also an ff-(mono)morphism;
	\item[(2)] if $nm$ is an ff-(mono)morphism, then $m$ is an ff-(mono)morphism;
	\item[(3)] the 2-pullback of an ff-(mono)morphism is an ff-(mono)morphism.
\end{itemize}
\end{lemma}

\begin{definition}\label{comma obj def}
Given an ordered pair of morphisms $(f\colon X\to Y, g\colon Z\to Y)$ in an $\Ord$-category $\C$ with common codomain, the \defn{(strict) comma object} of $(f,g)$ is defined by an object $f/g$ and morphisms $\pi_1\colon f/g\to X$, $\pi_2\colon f/g\to Z$ (also called ``projections'') such that
\begin{itemize}
	\item[(C1)] $f\pi_1\moi g\pi_2$;
	\item[(C2)] it has the universal property: given morphisms $\alpha\colon A\to X$ and $\beta\colon A\to Z$ such that $f\alpha\moi g\beta$, there exists a unique morphism $\la \alpha,\beta\ra\colon A\to f/g$ such that $\pi_1\la \alpha,\beta\ra =\alpha$ and $\pi_2\la \alpha,\beta\ra = \beta$ (see diagram \eqref{comma f/g} below);
	\item[(C3)] for morphisms $\alpha,\alpha'\colon A\to X$, $\beta,\beta'\colon A\to Z$ such that $f\alpha\moi g\beta$, $f\alpha'\moi g\beta'$, $\alpha\moi \alpha'$ and $\beta\moi \beta'$, the corresponding unique morphisms $\la \alpha,\beta\ra,\la \alpha',\beta'\ra\colon A\to f/g$ verify $\la \alpha,\beta\ra\moi \la \alpha',\beta'\ra$;
\end{itemize}
\end{definition}

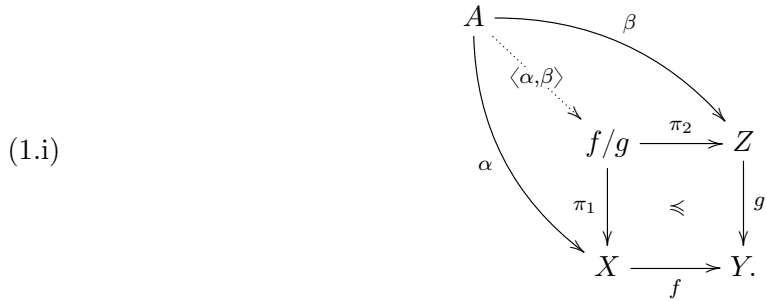
\begin{equation}
\label{comma f/g}
\vcenter{\xymatrix@=30pt{ A \ar@/_15pt/[ddr]_-{\alpha} \ar@/^15pt/[drr]^-{\beta} \ar@{.>}[dr]|-{\la \alpha,\beta\ra} \\
								 & f/g \ar[r]^-{\pi_2} \ar[d]_-{\pi_1} \ar@{}[dr]|-{\moi} & Z \ar[d]^-g \\
								 & X \ar[r]_-f & Y.}}
\end{equation}
From conditions (C2) and (C3) we can deduce that $(\pi_1,\pi_2)$ is \emph{jointly ff-monomorphic}. If $\C$ admits 2-products, this translates into the fact that $\la \pi_1,\pi_2\ra\colon f/g\mono X\times Z$ is an ff-monomorphism.

The following result combines comma objects and 2-pullbacks; a proof can be found, for instance, in \cite{Vassilis}:
\begin{lemma}
\label{2 comma. 12 comma <=> 1 pb}
Let $\C$ be an $\Ord$-category. Consider the diagram
\[
	\xymatrix{P \ar[d]_-{p_1} \ar[r]^-{p_2} & f/g \ar[d]_-{\pi_1} \ar[r]^-{\pi_2} \ar@{}[dr]|-{\moi} & Z \ar[d]^-g \\
						X' \ar[r]_-x & X \ar[r]_-f &  Y}
\]
where the right square is a comma object and the left square is commutative. The outer rectangle is a comma object if and only if the left square is a $2$-pullback.
\end{lemma}

Recall that a category $\Ccal$ is a 1-dimensional \defn{regular category}~\cite{Barr} when:
\begin{itemize}
	\item $\Ccal$ has finite limits;
	\item $\Ccal$ has coequalisers of kernel pairs;
	\item regular epimorphisms are stable under pullbacks in $\Ccal$.
\end{itemize}
\noindent Regular categories may also be characterised as categories $\Ccal$ such that (see~\cite{Borceux}, for example):
\begin{itemize}
	\item[(r1)] $\Ccal$ has finite limits;
	\item[(r2)] $\Ccal$ admits a (strong epimorphism, monomorphism) factorisation system;
	\item[(r3)] strong epimorphisms are stable under pullbacks in $\Ccal$;
	\item[(r4)] every strong epimorphism is a coequaliser (of its kernel pair).
\end{itemize}
\noindent When $\Ccal$ is a regular category the notions of regular epimorphism and strong epimorphism coincide, so that condition (r4) follows from (r1)-(r3).

This latter formulation led to the following definition of regular $\Pos$-category~\cite{KV}. A $\Pos$-category $\C$ is called \defn{regular} when:
\begin{itemize}
	\item[(R1)] $\C$ has finite (weighted) limits;
	\item[(R2)] $\C$ admits an (so-morphism, ff-monomorphism) factorisation system;
	\item[(R3)] so-morphisms are stable under 2-pullbacks in $\C$;
	\item[(R4)] every so-morphism is a coinserter (of its comma object).
\end{itemize}
\noindent As in the 1-dimensional case, condition (R4) follows from (R1)-(R3), as shown in~\cite{Vassilis}.

There are several approaches to regularity for enriched categories (see for instance~\cite{BourkeGarner}), but here we chose to use the following one, which allows for an easy calculus of (order) ideals.

\begin{definition}\label{enriched regular cat}
An $\Ord$-category $\C$ is called \defn{regular} when:
\begin{itemize}
	\item[(R1)] $\C$ has finite (weighted) limits;
	\item[(R2)] $\C$ admits an (so-morphism, ff-monomorphism) factorisation system;
	\item[(R3)] so-morphisms are stable under 2-pullbacks in $\C$;
	\item[(R4)] every so-morphism is a bicoinserter (of its comma object).
\end{itemize}
\end{definition}

Let us recall the ingredients needed for this concept. In the following $\C$ denotes an $\Ord$-category. The $\Ord$-enriched version of a monomorphism is that of an ff-monomorphism. The $\Ord$-enriched version of a strong epimorphism is defined next. A morphism $e\colon A\to B$ is called \defn{surjective on objects}, or \defn{so-morphism}, when $e$ is left orthogonal to every ff-monomorphism $m$, i.e. the usual diagonal fill-in property holds
\[
	\xymatrix{A \ar[r]^-e \ar[d]_-u & B \ar[d]^-v \ar@{.>}[dl]|-d \\ X \ar@{ >->}[r]_-m & Y.}
\]

If $\C$ has binary 2-products, then it is easy to check that every so-morphism is necessarily an epimorphism. We denote so-morphisms with arrows of the type $\repi$.

\begin{lemma}\label{pps for so-ms}
Let $e\colon A\to B$ and $f\colon B\to C$ be morphisms in an $\Ord$-category $\C$. Then:
\begin{itemize}
	\item[(1)] if $e$ and $f$ are so-morphisms, then $fe$ is also an so-morphism;
	\item[(2)] if $fe$ is an so-morphism, then $f$ is an so-morphism.
\end{itemize}
\end{lemma}

Let $\C$ be an $\Ord$-category. The \defn{coinserter} of a pair of parallel morphisms $a,b\colon X\to A$ is a morphism $c\colon A\to B$ such that
\begin{itemize}
	\item $ca\moi cb$;
	\item if $d\colon A\to D$ is such that $da\moi db$, then there exists a unique morphism $\lambda\colon B \to D$  such $d=\lambda c$;
	\item given a morphism $d'\colon A\to D$ such that $d'a\moi d'b$ and $d\moi d'$, then the corresponding unique morphisms $\lambda,\lambda'\colon B \to D$ verify $\lambda\moi \lambda'$;
\[
\xymatrix@C=20pt@R=10pt{ X \ar@<2pt>[r]^-a \ar@<-2pt>[r]_-b & A \ar[rr]^-c \ar[dr]_-d & \ar@{}[d]|(.4){\circlearrowright}& B. \ar@{.>}[dl]^-{\lambda}\\ & & D}
\]
\end{itemize}

Passing from $\Pos$-enriched categories to $\Ord$-enriched categories, where antisymmetry is no longer assumed, coinserter is no longer the key notion, being replaced by the notion of bicoinserter. The \defn{bicoinserter} of $a,b$ is a morphism $e\colon A\to E$ such that
\begin{itemize}
	\item $ea\moi eb$;
	\item if $f\colon A\to F$ is such that $fa\moi fb$, then there exists a morphism $t\colon E \to F$  such $f\cong te$, i.e. $f\moi te$ and $te\moi f$;
	\item given a morphism $f'\colon A\to F$ such that $f'a\moi f'b$ and $f\moi f'$, then the corresponding existing morphisms $t,t'\colon E \to F$ verify $t\moi t'$;
\[
\xymatrix@C=20pt@R=10pt{ X \ar@<2pt>[r]^-a \ar@<-2pt>[r]_-b & A \ar[rr]^-e \ar[dr]_-f & \ar@{}[d]|(.3){\cong} & E. \ar[dl]^-{t}\\ & & F}
\]
\end{itemize}
\noindent It is easy to check that a coinserter is always a bicoinserter.

\begin{remark}\label{Remark on (R4)} 
\begin{itemize}
	\item It is easy to prove that, in a regular $\Ord$-category, every bicoinserter $e\colon A\to E$ will be the bicoinserter of its comma object projections $(\pi_1\colon e/e \to A, \pi_2\colon e/e \to A)$.
	\item In a $\Pos$-enriched context, the notions of coinserter and bicoinserter coincide. It was shown in Proposition 2.16 of~\cite{Vassilis} that condition (R4) follows from (R1), (R2) and (R3). The proof is based on the fact that a certain morphism is an so-morphism and an ff-morphism, thus an isomorphism in that context. This argument cannot be used in the $\Ord$-enriched case, since an ff-morphism may not be a monomorphism. We do not know whether (R4) follows from (R1), (R2) and (R3) in the $\Ord$-enriched context.
\end{itemize}
\end{remark}

\begin{remark}\label{remark on Ord-quasiregular} One of the aims of this work is to develop the calculus of relations in an $\Ord$-enriched context. To do so we only need an $\Ord$-category for which conditions (R1), (R2) and (R3) hold. Therefore, in the next sections the assumption that ``$\C$ is a regular $\Ord$-category'' can be replaced by  ``$\C$ is an $\Ord$-category which satisfies conditions (R1), (R2) and (R3)''. The only result that depends on condition (R4) is the one stated in Theorem~\ref{R = f_*}. 
\end{remark}

\begin{example}\label{exs of Ord-regular cats}
\begin{enumerate}
	\item[1.] $\Ord$ is a regular $\Ord$-category.
	\item[2.] Every 1-dimensional regular category equipped with the discrete order is a regular $\Ord$-category.
	\item[3.] Any variety of ordered algebras~\cite{BloomWright} is a regular $\Ord$-category.
\end{enumerate}
\end{example}

The (so-morphism, ff-monomorphism) factorisation system is stable under 2-pull\-backs in $\C$. Actually, in a regular $\Ord$-category $\C$, so-morphisms are also stable under comma objects, as we show in the following lemmas.

\begin{lemma}\label{f/1_Y, 1_Y/f and their projs} Let $\C$ be an $\Ord$-category which admits comma objects. Consider the comma objects $f_*=f/1_Y$, $f^*=1_Y/f$ and the induced morphisms $\lambda=\la 1_X,f\ra$ and $\mu=\la f,1_X\ra$
\[
\xymatrix{X \ar@{=}@/_15pt/[ddr]_-{1_X} \ar@/^15pt/[drr]^-f \ar@{.>}[dr]|-{\lambda}\\
					& f_*=f/1_Y \ar[r]^-{\pi_Y} \ar[d]_-{\pi_X} \ar@{}[dr]|-{\moi} & Y \ar@{=}[d]^-{1_Y} \\
					& X \ar[r]_-f & Y.}
\;\;\begin{array}{c}\vspace{65pt} \\\mathrm{and} \end{array}\;\;\;\;\;\;
\xymatrix{X \ar@/_15pt/[ddr]_-{f} \ar@{=}@/^15pt/[drr]^-{1_X} \ar@{.>}[dr]|-{\mu}\\
					& f^*=1_Y/f \ar[r]^-{\rho_X} \ar[d]_-{\rho_Y} \ar@{}[dr]|-{\moi} & X \ar[d]^-{f} \\
					& Y \ar@{=}[r]_-{1_Y} & Y.}
\]
The projections $\pi_X$ and $\rho_X$ are split epimorphisms (thus, they are so-morphisms). If $f$ is an so-morphism, then so are $\pi_Y$ and $\rho_Y$.
\end{lemma}
\begin{proof} The proof is straightforward, and uses Lemma~\ref{pps for so-ms}.
\end{proof}

\begin{lemma}\label{f/g and preserves so-ms} Let $\C$ be a regular $\Ord$-category. Then so-morphisms are stable under comma objects. \end{lemma}
\begin{proof}
Consider a comma object $f/g$
\[
\xymatrix{f/g \ar[r]^-{\pi_2} \ar[d]_-{\pi_1} \ar@{}[dr]|-{\moi} & Z \ar@{>>}[d]^-{g} \\
					X \ar[r]_-f & Y,}
\]
where $g$ is an so-morphism. Consider the diagram
\[
	\xymatrix@!0@=45pt{f/g \pullback \ar[r] \ar@{>>}[d]_-{\pi_1}  \ar@(ur,ul)[rr]^-{\pi_2} & g^* \ar[r] \ar@{>>}[d]_-{\rho_Y} \ar@{}[dr]|-{\moi} & Z \ar@{>>}[d]^-g \\
						X \ar[r]_-f & Y \ar@{=}[r]_-{1_Y} & Y,}
\]
where the left side is a 2-pullback. The outer rectangle is the comma object of $(f,g)$ by Lemma~\ref{2 comma. 12 comma <=> 1 pb}. Consequently, $\rho_Y$ is an so-morphism by Lemma~\ref{f/1_Y, 1_Y/f and their projs} and $\pi_1$ is an so-morphism, since $\C$ is regular (Definition~\ref{enriched regular cat}(R3)). A similar proof holds for $f$ and $\pi_2$.
\end{proof}

\begin{remark}\label{comma objs do not preserve ff-monos} When $\C$ is an $\Ord$-category with comma objects, ff-monomorphisms are not necessarily stable under comma objects in $\C$. This is easily seen by taking $g=1_Y$, as in Lemma~\ref{f/1_Y, 1_Y/f and their projs}.
\end{remark}

\section{Relations in the 1-dimensional regular context}\label{Relations in the regular context context}

In this section we recall the basic definitions concerning (internal) relations in a 1-dimensional category, which shall be denoted by $\Ccal$ (to distinguish it from the $\C$ which is used in an $\Ord$-enriched context). We aim to extend some of those notions and results to the $\Ord$-enriched context. To distinguish a relation in this usual sense from the one in the enriched context, we call the former an ``ordinary relation''.

Let $\Ccal$ be an arbitrary category. An \defn{ordinary relation} $R$ from an object $X$ to an object $Y$ of $\Ccal$ is a span $X\stackrel{r_1}{\longleftarrow} R \stackrel{r_2}{\longrightarrow} Y$ such that $(r_1,r_2)$ is jointly monomorphic. The \emph{opposite} relation of $R$, denoted $R^\circ$, is the span $Y\stackrel{r_2}{\longleftarrow} R \stackrel{r_1}{\longrightarrow} X$. If $\Ccal$ admits binary products, then an ordinary relation as above can be viewed as a monomorphism $\la r_1,r_2\ra \colon R\to X\times Y$. When $X=Y$, we simply say that $R$ is an ordinary relation on $X$.

Any morphism $x\colon A\to X$ of $\Ccal$ can be seen a ``generalised element'' of $X$. Given $x\colon A\to X$ and $y\colon A\to Y$, we write $(x,y)\in_A R$, or simply $x R y$ (omitting the domain of the morphisms when this is not relevant), when there exists a commutative diagram
\begin{equation}\label{in R}
	\vcenter{\xymatrix@R=10pt{ & X \\
		A \ar[ur]^-{x} \ar[dr]_-{y} \ar@{.>}[rr] & & R. \ar[ul]_-{r_1} \ar[dl]^-{r_2}\\
		& Y}}
\end{equation}

\begin{definition}\label{def of difunctional}
An ordinary relation $X\stackrel{d_1}{\longleftarrow} D \stackrel{d_2}{\longrightarrow} Y$ in $\Ccal$ is called \emph{difunctional} when the relation
\begin{equation}\label{W-difunctional}
  \xymatrix@C=40pt{\Ccal(W,X) & \ar[l]_-{\Ccal(W,d_1)} \Ccal(W,D) \ar[r]^-{\Ccal(W,d_2)} & \Ccal(W,Y)}
\end{equation}
in $\Set$ is difunctional~\cite{Riguet}, for every object $W$ of $\Ccal$. More precisely, given morphisms $x,u\colon W\to X$, $y,v\colon W\to Y$, we have $(xD y \wedge uDy \wedge uDv)\Rightarrow xDv$. This can be pictured as
\begin{equation}\label{picture of difunctional}
	\begin{array}{ccc}
	 x & D & y \\
	 u & D & y \\
	 u & D & v \\
	\hline
	x & D & v.
	\end{array}
\end{equation}
\end{definition}

The definition of a reflexive, symmetric, transitive, and equivalence ordinary relation in $\Ccal$ is obtained similarly.

In order to define the composition of ordinary relations, the right setting is that of a regular category. Let $\Ccal$ be a regular category and consider ordinary relations $\la r_1,r_2\ra \colon R \to X\times Y$ and $\la s_1,s_2\ra \colon S \to Y\times Z$. The composite ordinary relation $SR\to X\times Z$ is defined through the (regular epimorphism, monomorphism) factorisation of $\la r_1p_1, s_2p_2\ra$ in
\[
	\xymatrix@R=10pt{R\times_Y S \ar[rr]^-{\la r_1p_1, s_2p_2\ra} \ar[dr]_(.4){\mathrm{regular\;\; epi\;\;\;}} & & X\times Z, \\ & SR \ar[ur]_(.6){\mathrm{\;mono}}}
\]
given the pullback
\[
	\xymatrix@!0@=40pt{ & & R\times_Y S \vpullback \ar[dl]_-{p_1} \ar[dr]^-{p_2} \\
	& R \ar[dl]_-{r_1} \ar[dr]_-{r_2} & & S \ar[dl]^-{s_1} \ar[dr]^-{s_2} \\
	X & & Y & & Z.}
\]

\begin{lemma}[\cite{CKP}]\label{generalised els in ordinary SR}
Let $\Ccal$ be a regular category. Consider ordinary relations $R\to X\times Y$, $S\to Y\times Z$, and generalised elements $x\colon A\to X$, $z\colon A\to Z$. Then
$(x,z)\in_A SR$ if and only if there exists a regular epimorphism $b\colon B \to A$ and a morphism $y\colon B\to Y$ such that $(xb,y)\in_B R$ and $(y,zb)\in_B S$.
\end{lemma}

This lemma allows one to prove that, in a regular category $\Ccal$, the composition of relations is associative. We get a bicategory $\Rel(\Ccal)$ of ordinary relations in $\Ccal$:
\begin{itemize}
	\item a 0-cell in $\Rel(\Ccal)$ is an object of $\Ccal$;
	\item a 1-cell from $X$ to $Y$ is an ordinary relation $R\to X\times Y$, also denoted by $R\colon X \relto Y$;
	\item a 2-cell from $R$ to $R'$ is denoted by $R\subseteq R'$, and holds when $R$ factors through $R'$
	\begin{equation}\label{R <= R'}
		\vcenter{\xymatrix@R=10pt{R \ar[dr]_-{\la r_1,r_2\ra} \ar@{.>}[rr] & & R'. \ar[dl]^-{\la r'_1,r'_2\ra}\\
		& X\times Y}}
	\end{equation}
	We write $R\cong R'$ when $R\subseteq R'$ and $R'\subseteq R$;
	\item the identity 1-cell on $X$ is given by the \emph{identity ordinary relation} $\Delta_X=\la 1_X,1_X\ra \colon X\to X\times X$.
\end{itemize}
	
From~\cite{FreydScendrov1990}, $\Rel(\Ccal)$ is a tabular allegory, with anti-involution given by taking the opposite ordinary relation. Freyd's \emph{modular laws} hold: given ordinary relations $R\colon X \relto Y$, $S\colon Y \relto Z$ and $T\colon X \relto Z$ we have
\begin{equation}\label{Freyd1}
	SR\wedge T \subseteq S(R\wedge S^\circ T)
\end{equation}
and
\begin{equation}\label{Freyd2}
	SR\wedge T \subseteq (S\wedge TR^\circ)R.
\end{equation}

Given an arbitrary category $\Ccal$, any morphism $f\colon X\to Y$ of $\Ccal$ induces two ordinary relations $X\stackrel{1_X}{\longleftarrow} X \stackrel{f}{\longrightarrow} Y$, denoted by $f_\nt$, and $Y\stackrel{f}{\longleftarrow} X \stackrel{1_X}{\longrightarrow} X$, denoted by $f^\nt$. If $\Ccal$ is a regular category, for every morphism $f\colon X\to Y$ in $\Ccal$, $f_\nt$ is a \emph{map} (in the sense of Lawvere) in $\Rel(\Ccal)$, meaning that it admits a right adjoint $f^\circ$, i.e. $f_\nt \dashv f^\circ$, so that the inclusions $\Delta_X\subseteq f^\circ f_\nt$ and $f_\nt f^\circ \subseteq \Delta_Y$ hold in $\Rel(\Ccal)$. On the other hand, taking a map $X\stackrel{r_1}{\longleftarrow} R \stackrel{r_2}{\longrightarrow} Y$ in $\Rel(\Ccal)$ guarantees that $r_1$ is a monomorphism and a regular epimorphism, which is necessarily an isomorphism; thus, $R\cong (r_2)_\nt$.

\begin{remark}\label{R=r_2r_1^op} Let $X\stackrel{r_1}{\longleftarrow} R \stackrel{r_2}{\longrightarrow} Y$ be an ordinary relation in a regular category $\Ccal$. It is easy to check that $R\cong (r_2)_\nt (r_1)^\circ$ and $R^\circ \cong (r_1)_\nt (r_2)^\circ$ ($R^\circ \cong ((r_2)_\nt (r_1)^\circ)^\circ\cong ((r_1)^\circ)^\circ (r_2)^\circ\cong (r_1)_\nt (r_2)^\circ$).
\end{remark}

\begin{remark}\label{difunctional in a regular cat} Let $\Ccal$ be a regular category. Then an ordinary relation $D\colon X \relto Y$ is difunctional (Definition~\ref{def of difunctional}) when $DD^\circ D\subseteq D$. Since $D\subseteq DD^\circ D$ always holds, $D$ is difunctional if and only $DD^\circ D\cong D$. Given any morphism $f\colon X\to Y$, one always has $f_\nt f^\circ f_\nt\cong f_\nt$ and $f^\circ f_\nt f^\circ \cong f^\circ$, which proves that $f_\nt$ and $f^\circ$ are examples of difunctional ordinary relations.
\end{remark}

\begin{remark}\label{refl, symm, trans in a regular cat} An ordinary relation $X\stackrel{r_1}{\longleftarrow} R \stackrel{r_2}{\longrightarrow} X$ in a category $\Ccal$ is:
\begin{itemize}
	\item reflexive when $(1_X,1_X)\in_X R$, meaning that there exists a morphism $e\colon X\to R$ such that $r_1e=1_X=r_2e$; equivalently $\Delta_X\subseteq R$;
	\item symmetric when $(r_2,r_1)\in_R R$, meaning that there exists a morphism $s\colon R\to R$ such that $r_1s=r_2$ and $r_2s=r_1$; also $R^\circ\subseteq R$ or, equivalently, $R^\circ \cong R$;
\end{itemize}
If $\Ccal$ is a regular category, so that composition of ordinary relations exists, $R$ is:
\begin{itemize}
	\item transitive when $RR\subseteq R$;
	\item an ordinary equivalence relation when it is reflexive, symmetric, and transitive, so that  $\Delta_X\subseteq R$, $R^\circ \cong R$ and $RR\cong R$ ($R\subseteq RR$ follows from the reflexivity of $R$).
\end{itemize}
\end{remark}

There are many other properties concerning (the calculus of) ordinary relations which can be found in~\cite{CKP}. Instead of recalling them all here, we focus on their generalisations to the context of (regular) $\Ord$-enriched categories next.

\section{Relations in the \texorpdfstring{$\Ord$}{Ord}-enriched context}\label{Relations in the Ord-enriched context}

We extend the content of Section~\ref{Relations in the regular context context} to the enriched context. We shall use the same names and notation whenever it is possible. In this section $\C$ denotes an $\Ord$-category.

A \defn{relation} from an object $X$ to an object $Y$ of $\C$ is a span $X\stackrel{r_1}{\longleftarrow} R \stackrel{r_2}{\longrightarrow} Y$ such that $(r_1,r_2)$ is jointly ff-monomorphic. The \emph{opposite} relation $R^\circ$ is the span $Y\stackrel{r_2}{\longleftarrow} R \stackrel{r_1}{\longrightarrow} X$. If $\C$ admits binary 2-products, then a relation is given by an ff-monomorphism $\la r_1,r_2\ra \colon R\mono X\times Y$. When $X=Y$, we simply say that $R$ is a relation on $X$.

Given morphisms $x\colon A\to X$ and $y\colon A\to Y$ of $\C$, we use the same notation $(x,y)\in_A R$, or $xRy$, when there exists a factorisation as in \eqref{in R}.

\begin{example}\label{any f/g is a relation}
\begin{enumerate}
\item[1.]
Any comma object in an $\Ord$-category $\C$
\[
\xymatrix@=30pt{ f/g \ar[r]^-{\pi_2} \ar[d]_-{\pi_1} \ar@{}[dr]|-{\moi} & Z \ar[d]^-g \\
						      X \ar[r]_-f & Y,}
\]
gives a relation $X\stackrel{\pi_1}{\longleftarrow} f/g \stackrel{\pi_2}{\longrightarrow} Z$ since $( \pi_1,\pi_2)$ is jointly ff-monomorphic by Definition~\ref{comma obj def}. Moreover, given generalised elements $x\colon A\to X$, $z\colon A\to Z$, we have
\begin{equation}\label{in f/g}
 (x,z)\in_A f/g \Leftrightarrow fx\moi gz.
\end{equation}
\item[2.] When $f=g=1_X$, we write $I_X=1_X/1_X$ and denote its projections by $x_1,x_2\colon I_X\to X$. Given two generalised elements $x,x'\colon A\to X$, we have $(x,x')\in_A I_X$ if and only if $x\moi x'$.
\end{enumerate}
\end{example}

To define the composition of relations, we must assume $\C$ to be a regular $\Ord$-category. Given relations $\la r_1,r_2\ra \colon R \mono X\times Y$ and $\la s_1,s_2\ra \colon S \mono Y\times Z$, the composite relation $SR\mono X\times Z$ is defined through the (so-morphism, ff-monomorphism) factorisation of $\la r_1p_1, s_2p_2\ra$ in
\[
	\xymatrix@R=10pt{R\times_Y S \ar[rr]^-{\la r_1p_1, s_2p_2\ra} \ar@{->>}[dr]_(.4){\text{so-morphism\;\;\;}} & & X\times Z, \\ & SR \ar@{ >->}[ur]_(.6){\text{\;\;\;\;\;\;\;ff-monomorphism}}}
\]
given the 2-pullback
\[
	\xymatrix@!0@=40pt{ & & R\times_Y S \vpullback \ar[dl]_-{p_1} \ar[dr]^-{p_2} \\
	& R \ar[dl]_-{r_1} \ar[dr]_-{r_2} & & S \ar[dl]^-{s_1} \ar[dr]^-{s_2} \\
	X & & Y & & Z.}
\]

The $\Ord$-enriched version of Lemma~\ref{generalised els in ordinary SR} holds in $\C$, with the difference that ``regular epimorphism'' is now ``so-morphism''.

\begin{lemma}[\cite{Vassilis}]\label{generalised els in SR}
Let $\C$ be a regular $\Ord$-category. Consider relations $R\mono X\times Y$, $S\mono Y\times Z$, and generalised elements $x\colon A\to X$, $z\colon A\to Z$. Then
$(x,z)\in_A SR$ if and only if there exists an so-morphism $b\colon B \repi A$ and a morphism $y\colon B\to Y$ such that $(xb,y)\in_B R$ and $(y,zb)\in_B S$.
\end{lemma}

This lemma allows one to prove that, in a regular $\Ord$-category $\C$, the composition of relations is associative. We get a bicategory $\Relf(\C)$ of relations in $\C$:
\begin{itemize}
	\item a 0-cell in $\Relf(\C)$ is an object of $\C$;
	\item a 1-cell from $X$ to $Y$ is a relation $R\mono X\times Y$;
	\item a 2-cell from $R$ to $R'$ is denoted by $R\subseteq R'$, and holds when $R$ factors through $R'$ as in \eqref{R <= R'};
	\item the identity 1-cell on $X$ is given by the \emph{identity relation} $\Delta_X=\la 1_X,1_X\ra \colon X\mono X\times X$.
\end{itemize}
	
From~\cite{FreydScendrov1990}, $\Relf(\C)$ is a tabular allegory, with anti-involution given by taking the opposite relation. Freyd's \emph{modular laws} still hold in $\Relf(\C)$: see \eqref{Freyd1} and \eqref{Freyd2}.

\section{Order ideals and their calculus of relations}\label{Order ideals and their calculus of relations}

The bicategory $\Relf(\C)$ does not capture entirely the enriched features of a regular $\Ord$-category. To do so, we shall consider relations with a kind of ``compatibility'' condition. Such relations were called weakening-closed in~\cite{Kurz2023, Vassilis}. We prefer to follow~\cite{CarboniStreet} and call them (order) ideals.

\begin{definition}\label{w-c relation} A relation $X\stackrel{r_1}{\longleftarrow} R \stackrel{r_2}{\longrightarrow} Y$ in an $\Ord$-category $\C$ is called an \defn{ideal} when, given generalised elements $x,x'\colon A\to X$, $y,y'\colon A\to Y$, we have
\begin{equation}\label{w-c def}
	\left( x'\moi x \;\wedge\; (x,y)\in_A R \;\wedge\; y\moi y' \right) \Rightarrow (x',y')\in_A R.
\end{equation}
\end{definition}
\noindent Note that an ideal $X\stackrel{r_1}{\longleftarrow} R \stackrel{r_2}{\longrightarrow} Y$ is a relation, by definition. So, $(r_1,r_2)$ is jointly ff-monomorphic. We use the notation $R\colon X\looparrowright Y$ for ideals.

\begin{example}\label{any f/g is w-c relation}
Any comma object in an $\Ord$-category $\C$
\[
\xymatrix@=30pt{ f/g \ar[r]^-{\pi_2} \ar[d]_-{\pi_1} \ar@{}[dr]|-{\moi} & Z \ar[d]^-g \\
						      X \ar[r]_-f & Y,}
\]
gives an ideal $X\stackrel{\pi_1}{\longleftarrow} f/g \stackrel{\pi_2}{\longrightarrow} Z$. Indeed, we already know that it is a relation by Example~\ref{any f/g is a relation}. Also, given generalised elements $x,x'\colon A\to X$, $z,z'\colon A\to Z$ such that $x'\moi x$, $(x,z)\in_A f/g$ and $z\moi z'$, then
\[ fx'\moi fx\stackrel{\eqref{in f/g}}{\moi} gz\moi gz';\]
we get $(x',z')\in_A f/g$.
\end{example}

If $\C$ is a regular $\Ord$-category, the composition of ideals is still an ideal. We denote by $\Relw(\C)$ the bicategory of ideals in $\C$, where identities are the ideals $I_X$, for every object $X$ of $\C$. It is easy to check that a relation $R\mono X\times Y$ in $\Relf(\C)$ is an ideal if and only if $R\cong I_YRI_X$. Consequently, $I_Y R I_X$ is always an ideal; we call it the \emph{ideal generated by $R$}.

Given an $\Ord$-category $\C$, any morphism $f\colon X\to Y$ of $\C$ can be seen as a relation $X\stackrel{1_X}{\longleftarrow} X \stackrel{f}{\longrightarrow} Y$, which is not necessarily an ideal. However, we can associate to any morphism $f\colon X\to Y$ two canonical ideals $f_*=f/1_Y\colon X\looparrowright Y$ and $f^*=1_Y/f\colon Y\looparrowright X$. It is easy to check that $I_X\subseteq f^*f_*$ and $f_*f^*\subseteq I_Y$. So, any morphism $f\colon X\to Y$ gives rise to an adjunction $f_*\dashv f^*$ in $\Relw(\C)$. The converse also holds, as we show below in Theorem~\ref{R = f_*} (see~\cite[Theorem 3.8]{Vassilis} for the $\Pos$-enriched case).

The following results are easy to prove and most can be found in~\cite{Vassilis}. Those results which are provided with a proof are new. Results involving relations (which are not necessarily ideals) and ideals are meant to hold in $\Relf(\C)$; this is explicitly added after the result. All the other results only involving ideals hold in $\Relw(\C)$, without explicitly mentioning it.

\begin{lemma}\label{basic lemma for f*} Let $\C$ be a regular $\Ord$-category. Consider morphisms $f,h\colon X\to Y$ and $g\colon Y\to Z$. Then:
\begin{itemize}
	\item[(1)] $f_*\cong I_Yf_\nt$ and $f^*\cong f^\circ I_Y$ in $\Relf(\C)$;
	\item[(2)] $f_\nt\subseteq f_*$ and $f^\circ \subseteq f^*$ in $\Relf(\C)$;
	\item[(3)] $(gf)_* \cong g_* f_*$;
	\item[(4)] $(gf)^*\cong f^* g^*$;
	\item[(5)] $f\moi h \Leftrightarrow h_*\subseteq f_* \Leftrightarrow f^*\subseteq h^*$;
	\item[(6)] $I_X\subseteq f^*f_*$;
	\item[(7)] $f_*f^*\subseteq I_Y$;
	\item[(8)] $f^*f_*f^*\cong f^*$ and $f_*f^*f_*\cong f_*$.
\end{itemize}
\end{lemma}

From (3), (4) and (5) we get $\Ord$-enriched functors $(\;)_*:\C^\co \to \Relw(\C)$ and $(\;)^*:\C^{\mathrm{op}} \to \Relw(\C)$.

\begin{lemma}\label{lemma on g/f, ff-ms and so-ms} Let $\C$ be a regular $\Ord$-category. Consider morphisms $f\colon X\to Y$, $g\colon ~Z\to Y$ and $h \colon X\to  Z$. Then:
\begin{itemize}
	\item[(1)] $f/g\cong g^*f_*$;
	\item[(2)] $f$ is an ff-morphism if and only if $I_X\cong f/f \cong f^*f_*$;
	\item[(3)] if $\moi$ is a partial order, then $f$ is an ff-monomorphism if and only if $I_X\cong f/f \cong f^*f_*$;
	\item[(4)] if $f$ is an so-morphism, then $f_*f^*\cong I_Y$;
	\item[(5)] if $\moi$ is a partial order, then $f$ is an so-morphism if and only if $f_*f^*\cong I_Y$;
	\item[(6)] $\la f, h\ra$ is an ff-morphism if and only if ${f}^*{f}_*\wedge {h}^*{h}_*\cong I_X$;
	\item[(7)] if $\moi$ is a partial order, then $\la f, h\ra$ is an ff-monomorphism if and only if ${f}^*{f}_*\wedge {h}^*{h}_*\cong I_X$.
\end{itemize}
\end{lemma}
\begin{proof}
(4) It is already known that $f_*f^*\subseteq I_Y$. Let $\la y_1,y_2\ra \colon I_Y \mono Y\times Y$ represent the projections of $I_Y=1_Y/1_Y$. Then, $y_1\moi y_2$ (see \eqref{in f/g}). It is easy to see that $(f,1_X)\in_X f^*$ and $(1_X,f)\in_X f_*$, so that $(f,f)\in_X f_*f^*$. If $f$ is an so-morphism, it follows that $(1_Y,1_Y)\in_Y f_*f^*$, i.e. there exists a factorisation such as
\[
\xymatrix@R=10pt{ & Y \\
		Y \ar[ur]^-{1_Y} \ar[dr]_-{1_Y} \ar@{.>}[rr] & & f_*f^*. \ar[ul] \ar[dl]\\
		& Y}
\]
Precomposing the dotted morphism with $y_1$, we get $(y_1,y_1)\in_{I_Y} f_*f^*$. Since $y_1\moi y_2$ and $f_*f^*$ is an ideal, then $(y_1,y_2)\in_{I_Y} f_*f^*$; this means precisely that $I_Y\subseteq f_*f^*$.

(5) Suppose that $f_*f^*\subseteq I_Y$. Since $(1_Y,1_Y)\in_Y I_Y\cong f_*f^*$, there exist an so-morphism $z\colon Z\repi Y$ and a morphism $x\colon Z\to X$ such that $(z,x)\in_Z f^*$ and $(x,z)\in_Z f_*$ (Lemma~\ref{generalised els in SR}). Using \eqref{in f/g}, we get $z\moi fx$ and $fx\moi z$. Since $\moi$ is a partial order, we conclude that $z=fx$. From Lemma~\ref{pps for so-ms}(2) we conclude that $f$ is an so-morphism.
\end{proof}

The following results generalise known ones concerning calculus of ordinary relations (see~\cite{CKP}).

\begin{proposition}\label{pps on w-c relations}Let $\C$ be a regular $\Ord$-category. Consider ideals $R, S\colon X\looparrowright Y$, $T\colon A\looparrowright B$, and morphisms $g\colon B\to Y$, $f\colon A\to X$, $k\colon Y\to B$, $h\colon X\to A$. Then:
\begin{itemize}
	\item[(1)] $g^*(R\wedge S)\cong g^*R \wedge g^*S$;
	\item[(2)] $(R\wedge S)f_*\cong Rf_* \wedge Sf_*$;
	\item[(3)] $k_*(R\wedge S)\subseteq k_*R \wedge k_*S$;
	\item[(4)] $(R\wedge S)h^*\subseteq Rh^* \wedge Sh^*$;
	\item[(5)] $g_*Tf^* \subseteq R \Leftrightarrow T\subseteq g^*Rf_*$.
\end{itemize}
\end{proposition}
\begin{proof}
(1) From $g^*(R\wedge S)\subseteq g^* R$ and $g^*(R\wedge S)\subseteq g^* S$, we conclude that $g^*(R\wedge S)\subseteq g^* R \wedge g^* S$.

Conversely, suppose that $(x,u)\in_U g^* R \wedge g^* S$. By Lemma~\ref{generalised els in SR}, there exist so-morphisms $e\colon V\repi U$, $e'\colon V'\repi U$ and morphisms $y\colon V\to Y$, $y'\colon V'\to Y$ such that $(xe,y)\in_V R$, $(y,ue)\in_V g^*$, $(xe',y')\in_{V'} S$, $(y',ue')\in_{V'} g^*$. We use \eqref{in f/g} and the fact that $R$ and $S$ are ideals to conclude that
\[\begin{array}{l}
	((xe,y)\in_V R\;\;\mathrm{and}\;\; y\moi gue) \;\;\Rightarrow \;\; (xe, gue)\in_V R; \\
	((xe',y')\in_{V'} S\;\;\mathrm{and}\;\; y'\moi gue') \;\;\Rightarrow \;\; (xe', gue')\in_{V'} S.
\end{array}
\]
Since $e$ and $e'$ are so-morphisms, we get $(x,gu)\in_U R$ and $(x,gu)\in_U S$, so $(x,gu)\in_U R\wedge S$. Since $(gu,u)\in_U g^*$, we may conclude that $(x,u)\in g^*(R\wedge S)$.

The proofs of (2), (3) and (4) follow similar arguments. The fact that (3) and (4) are only inclusions is a consequence of the compatibility property of ideals \eqref{w-c def} that does not apply for the other inclusions.

(5) Suppose that $g_*Tf^*\subseteq R$. Since $T$ is an ideal, then $T\cong I_B T I_A$. We have
\[
 T\cong I_BTI_A \stackrel{\mathrm{Lemma}~\ref{basic lemma for f*}(6)}{\subseteq} g^*g_* T f^* f_* \stackrel{\mathrm{assumption}}{\subseteq} g^*Rf_*.
\]
For the converse, suppose that $T\subseteq g^*R f_*$. Then
\[
 g_* Tf^* \stackrel{\mathrm{assumption}}{\subseteq} g_*g^* R f_* f^* \stackrel{\mathrm{Lemma}~\ref{basic lemma for f*}(7)}{\subseteq} I_Y R I_X \cong R,
\]
because $R$ is an ideal.
\end{proof}

When $\C$ is a regular $\Ord$-category, we may write $R\cong (r_2)_\nt (r_1)^\circ$ and $R^\circ \cong (r_1)_\nt (r_2)^\circ$ in $\Relf(\C)$ (see Remark~\ref{R=r_2r_1^op}). In $\Relw(\C)$ the role of $(\;)_\circ$, $(\;)^\circ$ is played by $(\;)_*$ and $(\;)^*$. Indeed, given an ideal $\la r_1,r_2\ra\colon R\mono X\times Y$, it is easily checked that $R\cong (r_2)_*(r_1)^*$. So, for any relation $\la r_1,r_2\ra\colon R\mono X\times Y$, we use the notation $R_*=(r_2)_*(r_1)^*$, and in fact $R_*$ is the ideal generated by $R$:
\[
	I_Y R I_X\cong I_Y(r_2)_\circ (r_1)^\circ I_X\cong (r_2)_*(r_1)^*=R_*.
\]
Now, to obtain the ideal which plays the role of the opposite (ideal) of $R$ in $\Relw(\C)$, we define $R^*=(r_1)_*(r_2)^*$. In particular, if $R=f_\nt$, for a morphism $f\colon X\to Y$, then $R^*=(1_X)_*(f)^*\cong I_X f^*\cong f^*$.

\begin{lemma}\label{R^* smallest ideal contains R^op}Let $\C$ be a regular $\Ord$-category. Given a relation $R\mono X\times Y$, $R^*$ is the smallest ideal containing $R^\circ$.
\end{lemma}
\begin{proof} We have $R^*=(r_1)_*(r_2)^*\cong I_X(r_1)_\circ(r_2)^\circ I_Y\cong I_XR^\circ I_Y\cong (R^\circ)_*$.
\end{proof}

\begin{lemma}\label{R<=S implies R*<=S*}Let $\C$ be a regular $\Ord$-category. If $R\mono X\times Y$, $S\mono X\times Y$ are relations such that $R\subseteq S$ (in $\Relf(\C)$), then $R_*\subseteq S_*$ and $R^*\subseteq S^*$.
\end{lemma}

We can now state an $\Ord$-enriched version for Freyd's modular laws.

\begin{proposition}\label{Freyd's enriched modular laws}Let $\C$ be a regular $\Ord$-category. For ideals $R\colon X\looparrowright Y$, $S\colon Y\looparrowright Z$ and $T\colon X\looparrowright Z$ we have
\begin{equation}\label{Freyd1 enriched}
	SR\wedge T \subseteq S(R\wedge S^* T)
\end{equation}
and
\begin{equation}\label{Freyd2 enriched}
	SR\wedge T \subseteq (S\wedge TR^*)R.
\end{equation}
\end{proposition}
\begin{proof}These are immediate consequences of Freyd's modular laws \eqref{Freyd1} and \eqref{Freyd2} (in $\Relf(\C)$) and Lemma~\ref{R^* smallest ideal contains R^op}.
\end{proof}

\begin{proposition}\label{DD^op D is w-c} Let $\C$ be a regular $\Ord$-category and consider ideals $R\colon X \looparrowright Y$, $S \colon Z \looparrowright Y$ and $T \colon Z\looparrowright W$ in $\C$. Then $TS^\circ R\cong TS^* R$. In particular, $TS^\circ R$ is an ideal.
\end{proposition}
\begin{proof} From Lemma~\ref{R^* smallest ideal contains R^op} we have: $T S^* R\cong T I_Y S^\circ I_Z R\cong T S^\circ R$.
\end{proof}

\begin{proposition}\label{2-pb of ideal is an ideal} Consider an ideal $R\colon X\looparrowright Y$ and morphisms $f\colon U\to X$, $g\colon V\to Y$ in an $\Ord$-category $\C$. The relation $S\mono U\times V$ given by the $2$-pullback of $R\mono X\times Y$ along $f\times g$
\[
	\xymatrix{S\cong g^\circ R f_\nt \pullback \ar[r]^-h \ar@{ >->}[d]_-{\la s_1,s_2\ra} & R \ar@{ >->}[d]^-{\la r_1,r_2\ra} \\
		U\times V \ar[r]_-{f\times g} & X\times Y,}
\]
is an ideal. If $\C$ is a regular $\Ord$-category, then $g^\circ R f_\nt \cong g^* R f_*$. In particular, the inverse image of an ideal $T\colon X\looparrowright X$ by the morphism $f$, denoted $f^{-1}(T)$, is an ideal such that $f^{-1}(T) \cong f^* T f_*$.
\end{proposition}
\begin{proof} By Lemma~\ref{pps for ff-ms}(3) $S$ is a relation, and it is easy to check that $S\cong g^\circ R f_\nt$ in $\Relf(\C)$. To see that it is an ideal, suppose we have morphisms $u,u'\colon A\to U$ and $v,v'\colon A\to V$ such that $(u,v)\in_A S$, $u'\moi u$ and $v\moi v'$. Then, there exists a morphism $\alpha\colon A\to S$ such that $\la s_1,s_2\ra \alpha = \la u,v\ra$. We get the factorisation $\la r_1,r_2\ra h \alpha = \la fu,gv\ra$, which shows that $(fu,gv)\in_A R$. Since $fu'\moi fu$ and $gv\moi gv'$, then $(fu',gv')\in_A R$, since $R$ is an ideal. Consequently, there exists a morphism $\beta\colon A\to R$ such that $\la r_1,r_2\ra \beta = \la fu',gv'\ra = f\times g \la u',v'\ra$. The universal property of the above $2$-pullback then gives a morphism $\gamma\colon A\to S$ such that $\la s_1,s_2\ra \gamma = \la u',v'\ra$. This proves that $(u',v')\in_A S$.

For the second statement, we use Lemma~\ref{basic lemma for f*}(1) and the fact that $R$ is an ideal to get $g^*Rf_* \cong g^\circ I_Y R I_X f_\nt \cong g^\circ R f_\nt$.
\end{proof}

\begin{example}\label{example of direct image of ideal which is not an ideal}
Let $\C$ be a regular $\Ord$-category. Consider an ideal $S\colon U\looparrowright V$ and morphisms $f\colon U\to X$, $g\colon V\to Y$. The (so-morphism, ff-monomorphism) factorisation of $f\times g \la s_1,s_2\ra$
\[
	\xymatrix{S  \ar@{>>}[r]^-h \ar@{ >->}[d]_-{\la s_1,s_2\ra} & R \ar@{ >->}[d]^-{\la r_1,r_2\ra} \\
		U\times V \ar[r]_-{f\times g} & X\times Y,}
\]
gives a relation $R\mono X\times Y$ which is not necessarily an ideal. In particular, when $U=V$ and $f=g$, such a factorisation gives the \emph{direct image} of $S$ under $f$. Consequently, the direct image of an ideal is not necessarily an ideal.

Consider the example where $S=I_U$, $f=1_U$ and $g\colon U\to Y$ is any morphism in $\C$
\[
	\xymatrix{I_U  \ar@{>>}[r]^-h \ar@{ >->}[d]_-{\la u_1,u_2\ra} & R\cong g_\nt I_U \ar@{ >->}[d]^-{\la r_1,r_2\ra} \\
		U\times U \ar[r]_-{1_U\times g} & U\times Y.}
\]
It is easy to see that $R\cong g_\nt I_U$. To have morphisms $u\colon A \to U$, $y\colon A\to Y$ such that $(u,y)\in_A g_\nt I_U$, means that there exist an so-morphism $b\colon B\repi A$ and a morphism $\overline{u}\colon B\to U$ such that
\[
\xymatrix@C=40pt{ B \ar@/^/[r]^-{ub} \ar@/_/[r]_-{\overline{u}} \ar@{}[r]|-{\rotatebox{-90}{$\moi$}} \ar@{>>}[d]_-b
	\ar@{}[dr]_(.6){\circlearrowright} & U \ar[d]^-g \\
	A \ar[r]_-y & Y;}
\]
here $g\overline{u}=yb$ (see Lemma~\ref{generalised els in ordinary SR}). If $g$ is an ff-monomorphism, then the  diagonal fill-in property gives a morphism $d\colon A\to U$ such that $y=gd$. We want to show that $g_\nt I_U$ is not an ideal. Given another morphism $y'\colon A\to Y$ such that $y\moi y'$, there is no reason why $y'$ should also factor through $g$. For example, take $g$ to be the ff-monomorphism $\la z_1,z_2\ra \colon I_Z\mono Z\times Z$, for some object $Z$, and $y=\la y_1,y_2\ra\colon A \to Z\times Z$. To have $y$ factor through $g$ implies that $(y_1,y_2)\in_A I_Z$, i.e. $y_1\moi y_2$. If $y'=\la y_1',y_2'\ra\colon A\to Z\times Z$ is another morphism such that $y\moi y'$, then $y_1\moi y_1'$ and $y_2\moi y_2'$. This does not imply that $y_1'\moi y_2'$, so that $y'$ may not factor through $g$.

A similar argument holds for $S=I_U$ and $f=g\colon U\to Y$. In that case $R\cong g_\nt I_U g^\circ$. To have $(u,y)\in_A g_\nt I_U g^\circ$ when $g$ is an ff-monomorphism still implies that $y$ factors through $g$. This shows that the direct image of an ideal is not necessarily an ideal.
\end{example}

\begin{theorem} \label{R = f_*}
Let $\C$ be a regular $\Ord$-category. If $R\colon X\looparrowright Y$ has a right adjoint $\overline{R}\colon Y\looparrowright X$ in $\Relw(\C)$, then there exists a unique morphism $f\colon X\to Y$ in $\C$ such that $R\cong f_*$.
\end{theorem}
\begin{proof} We partially follow the proof of ~\cite[Theorem 3.8.]{Vassilis}. Since $R$ has a right adjoint $\overline{R}$, then $I_X\subseteq \overline{R}R$ and $R\overline{R}\subseteq I_Y$. Consider the 2-pullback of $R$ and $(\overline{R})^\circ$
\[
	\xymatrix@R=30pt@C=40pt{S \pullback \ar@{ >->}[r]^-{\overline{u}} \ar@{ >->}[d]_-{u}  & \overline{R} \ar@{ >->}[d]^-{\la \overline{r}_2, \overline{r}_1\ra} \\ R \ar@{ >->}[r]_-{\la r_1,r_2\ra} & X\times Y.}
\]
Note that the relation $\la r_1,r_2\ra u\colon S \to X\times Y$ is the intersection of $R$ and $\overline{R}^\circ$, i.e. $S\cong R\cap \overline{R}^\circ$. 

We prove that $r_1u$ is an so-morphism. We have $(1_X,1_X)\in_X I_X\subseteq \overline{R}R$, so there exist an so-morphism $b\colon B\repi X$ and a morphism $y\colon B\to Y$ such that $(b,y)\in_B R$ and $(y,b)\in_B \overline{R}$ (see Lemma~\ref{generalised els in ordinary SR}). We get commutative diagrams
\[
\xymatrix@R=10pt{ & X \\
		B \ar@{>>}[ur]^-{b} \ar[dr]_-{y} \ar@{.>}[rr]^-{\alpha} & & R \ar[ul]_-{r_1} \ar[dl]^-{r_2}\\
		& Y} \hspace{20pt}
\xymatrix@R=10pt{ & Y \\
		A \ar[ur]^-{y} \ar@{>>}[dr]_-{b} \ar@{.>}[rr]^-{\overline{\alpha}} & & \overline{R}. \ar[ul]_-{\overline{r}_1} \ar[dl]^-{\overline{r}_2}\\
		& X}
\]
The universal property of the above pullback induces a unique morphism $\lambda:B\to S$ such that $u\lambda=\alpha$ and $\overline{u}\lambda =\overline{\alpha}$. Since $r_1 u\lambda=r_1\alpha=b$ is an so-morphism, then so is $r_1 u$ (see Lemma~\ref{pps for so-ms}(2)). 

From (R4) it follows that $r_1 u$ is the bicoinserter of the projections $(\pi_1,\pi_2)$ of its comma object (see Remark~\ref{Remark on (R4)})
\[
\xymatrix{C \ar[r]^-{\pi_2} \ar[d]_-{\pi_1} \ar@{}[dr]|-{\moi} & S \ar@{>>}[d]^-{r_1u} \\
					S \ar@{>>}[r]_-{r_1u} & X.}
\]
We have $\la r_1,r_2\ra u \pi_1 \moi \la r_1,r_2\ra u \pi_2$, since: \\
\begin{itemize}
	\item we already know that $r_1u\pi_1\moi r_1u\pi_2$;
	\item to prove that $r_2u\pi_1\moi r_2 u \pi_2$, we use the following argument. We have $(\overline{r}_1\overline{u}\pi_1, \overline{r}_2\overline{u}\pi_1)\in_C \overline{R}$, i.e. $(r_2u\pi_1, r_1u\pi_1)\in_C \overline{R}$. On the other hand, from $(r_1 u\pi_2,r_2u\pi_2)\in_C R$ and $r_1u \pi_1\moi r_1u\pi_2$ we get $(r_1u\pi_1,r_2u\pi_2)\in_C R$, since $R$ is an ideal. This gives $(r_2u\pi_1,r_2u\pi_2)\in_C R\overline{R}\subseteq I_Y$; consequently, $r_2u\pi_1\moi r_2u\pi_2$ (see Example~\ref{any f/g is a relation}2.).
\end{itemize}
	
Since $r_1 u$ is the bicoinserter of $(\pi_1,\pi_2)$, there exists a morphism $t$ in the diagram
\[
\vcenter{\xymatrix@C=20pt@R=10pt{ C \ar@<2pt>[r]^-{\pi_1} \ar@<-2pt>[r]_-{\pi_2} & S \ar@{ >->}[dr]_-u \ar@{>>}[rr]^-{r_1 u} & & X \ar[dd]^-{t=\la i,f\ra} \ar@{}[dl]|{\cong}  \\
	& & R \ar@{ >->}[dr]_-{\la r_1,r_2\ra} \\ 
	& & & X\times Y}}
\]
such that $\la r_1,r_2\ra u \cong tr_1 u$. We deduce that $r_2u\moi fr_1u$ and $fr_1u\moi r_2u$. From $fr_1u\moi r_2u$, we conclude that $(r_1u,r_2u)\in_S f_*$ by \eqref{in f/g}; this proves that $S\subseteq f_*$.

To finish, we prove that $R\cong f_*$. On one hand we have 
\[	
\begin{array}{rl}
f_* & \cong  f_*I_X \stackrel{\text{Lemma}~\ref{lemma on g/f, ff-ms and so-ms}(4)}{\cong} f_*(r_1 u)_* (r_1 u)^* \stackrel{\text{Lemma}~\ref{basic lemma for f*}(3)}{\cong} (fr_1 u)_*(r_1 u)^*  \\
	& \stackrel{\text{Lemma}~\ref{basic lemma for f*}(5)}{\subseteq} (r_2 u)_*(r_1 u)^* \cong S_* \stackrel{\text{Lemma}~\ref{R<=S implies R*<=S*}}{\subseteq} R_*\cong R.
\end{array}
\]

On the other hand we use $S\subseteq f_*$ to prove
\[
	\Delta_X\cong I_X\cap \Delta_X \subseteq \overline{R}R\cap \Delta_X  \stackrel{\eqref{Freyd1}}{\subseteq}  \overline{R}(R\cap \overline{R}^{\circ}) \cong \overline{R}S\subseteq \overline{R}f_*.
\]
We then deduce that 
\[
	R\cong R\Delta_X\subseteq R\overline{R}f_* \subseteq I_Yf_* \cong f_*.
\]
\end{proof}

\section{Mal'tsev property in the \texorpdfstring{$\Ord$}{Ord}-enriched context}\label{Mal'tsev property in the Ord-enriched context}

Recall from~\cite{CLP,CPP,CKP} that a Mal'tsev category $\Ccal$ is defined through the property that every ordinary reflexive relation $R\colon X\relto X$ in $\Ccal$ is an equivalence relation. The original definition asks for the base category $\Ccal$ to be regular or just finitely complete. However, this property on ordinary reflexive relations can be stated through generalised elements without any assumption on $\Ccal$, as in Definition~\ref{def of difunctional}; see also \eqref{picture of difunctional}. Actually, the difunctionality of ordinary relations is another equivalent way to define a Mal'tsev category.

\begin{definition}[see~\cite{CLP,CPP,CKP}]\label{Mal'tsev cat with no lims} A category $\Ccal$ is called a \defn{Mal'tsev category} when every ordinary relation $D\colon X\relto Y$ is difunctional.
\end{definition}

We recall Theorem 3.6 from~\cite{CKP}, which gives several well-known characterisations of a regular Mal'tsev category.

\begin{theorem}\label{from CKP} Let $\Ccal$ be a regular category. The following statements are equivalent, and characterise regular Mal'tsev categories:
\begin{itemize}
	\item[(i)] for any ordinary equivalence relations $R,S\colon X\relto X$ on an object $X$, $RS$ is an ordinary equivalence relation on $X$;
	\item[(ii)] $RS\cong SR$, for any ordinary equivalence relations $R,S\colon X\relto X$ on an object $X$;
	\item[(iii)] $RS\cong SR$ for any ordinary effective equivalence relations (i.e. kernel pairs of some morphism in $\Ccal$) $R,S\colon X\relto X$;
	\item[(iv)] every ordinary relation $D\colon X\relto Y$ is difunctional, i.e. $DD^\circ D\cong D$;
	\item[(v)] every ordinary reflexive relation $R\colon X \relto X$ on an object $X$ is an ordinary equivalence relation;
	\item[(vi)] every ordinary reflexive relation $R\colon X \relto X$ on an object $X$ is transitive;
	\item[(vii)] every ordinary reflexive relation $R\colon X \relto X$ on an object $X$ is symmetric.
\end{itemize}
\end{theorem}

\begin{definition}\label{Ord-Mal'tsev} An $\Ord$-category $\C$ is called an \defn{$\Ord$-Mal'tsev category} when every ideal $D\colon X\looparrowright Y$ satisfies the property: given morphisms $x, u, u'\colon A\to X$, $y,y',v\colon A\to Y$ such that $(x,y)\in_A D$, $y\moi y'$, $(u,y')\in_A D$, $u\moi u'$ and $(u',v)\in_A D$, then $(x,v)\in_A D$. This property may be pictured as
\begin{equation}\label{picture Ord-difunctional}
	\begin{array}{ccc}
	 x & D & y \vspace{-10pt}\\
	   &   & \rotatebox{-90}{$\moi$} \\
	 u & D & y' \vspace{-10pt}\\
	\rotatebox{-90}{$\moi$} \\
	 u' & D & v \\
	\hline
	x & D & v.
	\end{array}
\end{equation}
\end{definition}

\begin{proposition}\label{Mal'tsev => 2-Mal'tsev} Let $\Ccal$ be a Mal'tsev category. Then any $\Ord$-enrichment of $\Ccal$ is an $\Ord$-Mal'tsev category.
\end{proposition}
\begin{proof} Any relation $D\mono X\times Y$ is difunctional. If $D$ is an ideal, with the relations given in the top part of \eqref{picture Ord-difunctional}, we get $xDy'$ and $uDv$. Then
\[
	\begin{array}{ccc}
	 x & D & y'\\
	 u & D & y' \\
	 u & D & v \\
	\hline
	x & D & v
	\end{array}
\]
from the difunctionality of $D$.
\end{proof}

If $\C$ is a regular $\Ord$-category (so that we can compose relations in $\C$) then the property expressed in \eqref{picture Ord-difunctional} has a similar interpretation as $DD^\circ D\cong D$, in the 1-dimensional regular context.

\begin{proposition}\label{picture Ord-difunctional <=> DD*D=D} Let $\C$ be a regular $\Ord$-category and consider an ideal $\la d_1,d_2\ra \colon D\mono X\times Y$ in $\C$. Then $D$ satisfies \eqref{picture Ord-difunctional} if and only if $D\cong DD^*D\cong DD^\circ D$.
\end{proposition}
\begin{proof} We have $DD^* D\cong DD^\circ D$ from Proposition~\ref{DD^op D is w-c}. Suppose now that $D$ satisfies \eqref{picture Ord-difunctional}. We always have $D\subseteq DD^\circ D \subseteq DD^*D$, by Lemma~\ref{R^* smallest ideal contains R^op}. Next, consider $x\colon A\to X$ and $y\colon A\to Y$ such that $(x,y)\in_A DD^*D$. Applying Lemma~\ref{generalised els in SR} twice, there exist so-morphisms $b\colon B \repi A$ and $c\colon C\repi B$, and morphisms $\overline{x}\colon B\to X$, $\overline{y}\colon C\to Y$ such that $(xbc,\overline{y})\in_C D$, $(\overline{y}, \overline{x}c)\in_C D^*$ and $(\overline{x}c, ybc)\in_C D$. Since $D^*=(d_1)_*(d_2)^*$, then there exists an so-morphism $w\colon W\repi C$ and a morphism $z\colon W\to D$ such that $(\overline{y}w,z)\in_W (d_2)^*$ and $(z,\overline{x}cw)\in_W (d_1)_*$, again by Lemma~\ref{generalised els in SR}. So, $\overline{y}w\moi d_2 z$ and $d_1 z\moi \overline{x}cw$ by \eqref{in f/g}. By assumption, we have
\[
	\begin{array}{ccc}
	 xbcw & D & \overline{y}w \vspace{-10pt}\\
	   &   & \rotatebox{-90}{$\moi$} \\
	 d_1 z & D & d_2 z \vspace{-10pt}\\
	\rotatebox{-90}{$\moi$} \\
	 \overline{x}cw & D & ybcw \\
	\hline
	xbcw & D & ybcw.
	\end{array}
\]
Using the fact that $bcw$ is an so-morphism (Lemma~\ref{pps for so-ms}), we conclude that $(x,y)\in_A D$.

For the converse, suppose we have $DD^*D\subseteq D$. Consider generalised elements related as in \eqref{picture Ord-difunctional}. Using the fact that $D$ is an ideal and Lemma~\ref{R^* smallest ideal contains R^op}, we have
\[
\begin{array}{lcl}
	( (x,y)\in_A D \wedge y\moi y' ) & \Rightarrow & (x,y')\in_A D;\\
	(y',u)\in_A D^\circ \subseteq D^* & \Rightarrow & (y',u)\in_A D^*; \\
	( u\moi u' \wedge (u',v)\in_A D ) & \Rightarrow & (u,v)\in_A D.
\end{array}
\]
So, $(x,v)\in_A DD^*D\subseteq D$; thus $(x,v)\in_A D$.
\end{proof}

Let $\C$ be a regular $\Ord$-category. Next we show that the $\Ord$-enriched version of Theorem~\ref{from CKP} holds, thus giving several characterisations for regular $\Ord$-Mal'tsev categories. To do so we must study the enriched counterpart of an ordinary (effective) equivalence relation.

\begin{definition}[\cite{KV}]\label{congruence} Let $\C$ be an $\Ord$-category. An ideal $R\colon X\looparrowright X$ on an object $X$ which is reflexive and transitive is called a \defn{congruence} on $X$. A congruence $R$ is called \defn{effective} when $R\cong f/f\cong f^*f_*$, for some morphism $f\colon X\to Y$ in $\C$.
\end{definition}

\begin{lemma}[see~\cite{Vassilis}]\label{E reflexive <=> I_X subseteq E}
Let $\C$ be an $\Ord$-category. An ideal $R\colon X\looparrowright X$ is reflexive if and only if $I_X\subseteq R$. Consequently, $R$ is a congruence if and only if $I_X\subseteq R$ and $R$ is transitive.
\end{lemma}
\begin{proof}
If $R$ is reflexive, then $\Delta_X\subseteq R$. Hence, $I_X\cong I_X\Delta_X\subseteq I_XR\cong R$. For the converse, $\Delta_X\subseteq I_X\subseteq R$; thus $R$ is reflexive (see Remark~\ref{refl, symm, trans in a regular cat}).
\end{proof}

\begin{remark}\label{remark about congs not being symmetric} The notion of a congruence $R\colon X\looparrowright X$ does not involve any sort of symmetry for $R$. Symmetry of $R$ would mean that $R^\circ \cong R$, which is generally false for ideals (see Lemma~\ref{R^* smallest ideal contains R^op}). As a consequence, the $\Ord$-enriched version of Theorem~\ref{from CKP} does not include the statement (vii); also statements (v) and (vi) coincide. Actually, the symmetry of a reflexive and transitive ordinary relation comes for free when the base category is \emph{$n$-permutable}~\cite{M-FRVdL} (see~\cite{HM, CKP} for the definitions of $n$-permutable variety and $n$-permutable category). This is the case of Mal'tsev categories, which are $2$-permutable categories.
\end{remark}

\begin{theorem}\label{Ord-enriched for CKP} Let $\C$ be a regular $\Ord$-category. Then the following statements are equivalent and characterise regular $\Ord$-Mal'tsev categories:
\begin{itemize}
	\item[(i)] for any congruences $R,S\colon X\looparrowright X$ on an object $X$, $RS$ is a congruence on $X$;
	\item[(ii)] $RS\cong SR$, for any congruences $R,S\colon X\looparrowright X$ on an object $X$;
	\item[(iii)] $RS\cong SR$, for any effective congruences $R,S\colon X\looparrowright X$ on an object $X$;
	\item[(iv)] every ideal $X\stackrel{d_1}{\longleftarrow} D \stackrel{d_2}{\longrightarrow} Y$ is such that $DD^* D\cong D$;
	\item[(v)] every reflexive ideal $R\colon X\looparrowright X$ on an object $X$ is a congruence.
\end{itemize}
\end{theorem}
\begin{proof} Suppose that the reflexivity of $R$ and $S$ are given by the factorisations
\[
	\xymatrix@R=10pt{ & X \\
		X \ar[ur]^-{1_X} \ar[dr]_-{1_X} \ar@{.>}[rr]|-{e_R} & & R \ar[ul]_-{r_1} \ar[dl]^-{r_2}\\
		& X} \;\;\;\;\begin{array}{c} \vspace{33pt}\\ \mathrm{and} \end{array}\;\;\;\;
		\xymatrix@R=10pt{ & X \\
		X \ar[ur]^-{1_X} \ar[dr]_-{1_X} \ar@{.>}[rr]|-{e_S} & & S. \ar[ul]_-{s_1} \ar[dl]^-{s_2}\\
		& S}
\]
(i) $\Rightarrow$ (ii) If $RS$ is a congruence, then it is transitive: $RSRS\subseteq RS$. We use Proposition~\ref{pps on w-c relations}(5), Lemma~\ref{basic lemma for f*}(3) and (4) and the fact that $(1_X)^*\cong I_X\cong (1_X)_*$ is the identity in $\Relw(\C)$ to get the following
\[
\begin{array}{rcl}
	RSRS \subseteq RS & \Leftrightarrow & (r_2)_*(r_1)^*SR(s_2)_*(s_1)^* \subseteq RS \\
	& \Leftrightarrow & (r_1)^*SR(s_2)_* \subseteq (r_2)^* RS (s_1)_* \\
	& \Rightarrow & (e_R)^*(r_1)^*SR(s_2)_*(e_S)_* \subseteq (e_R)^*(r_2)^* RS (s_1)_*(e_S)_* \\
	& \Leftrightarrow & (r_1e_R)^*SR(s_2 e_S)_* \subseteq (r_2 e_R)^* RS (s_1e_S)_* \\
	& \Leftrightarrow & SR\subseteq RS.
\end{array}
\]
Similarly, we can obtain $RS\subseteq SR$; thus $RS\cong SR$.\\
(ii) $\Rightarrow$ (iii) This implication is obvious. \\
(iii) $\Rightarrow$ (iv) We have $DD^*D\cong (d_2)_*(d_1)^*(d_1)_*(d_2)^*(d_2)_*(d_1)^*$, by the definition of $D^*$. Since $(d_1)^*(d_1)_*$ and $(d_2)^*(d_2)_*$ are effective congruences, their composition commutes. We get
$DD^*D$ $\cong$ $(d_2)_* (d_2)^*(d_2)_* (d_1)^*(d_1)_*(d_1)^*\cong (d_2)_*(d_1)^*\cong D$ (see Lemma~\ref{basic lemma for f*}(8)).\\
(iv) $\Rightarrow$ (v) If $R$ is a reflexive relation, then so is $R^*$; thus $I_X\subseteq R^*$, be Lemma~\ref{E reflexive <=> I_X subseteq E}. We have to prove that $R$ is transitive: $RR\cong R I_X R\subseteq RR^*R\cong R$.\\
(v) $\Rightarrow$ (i) Since $R$ and $S$ are reflexive, so is the composite $RS$. By assumption, $RS$ is a congruence.
\end{proof}

\section{Examples of \texorpdfstring{$\Ord$}{Ord}-Mal'tsev categories}\label{Examples}

\begin{example}\label{Ex:any Mal'tsev cat}
Any $\Ord$-enrichment of a Mal'tsev category is an $\Ord$-Mal'tsev category by Proposition~\ref{Mal'tsev => 2-Mal'tsev}. In particular, the varieties of (abelian) groups, rings, modules over a ring, Boolean algebras, Heyting algebras are such. As non-varietal examples, we have the dual of any elementary (pre)topos or the category of topological groups. These (and more) examples can be found in~\cite{CKP,CPP,BB}.
\end{example}

\begin{example}\label{Ex:Mon with cancellation}
Let $\Mon_{lc}$ denote the category of monoids with \emph{left cancellation}. We use additive notation to denote such monoids, even though they are not necessarily abelian. By left cancellation we mean: $a+b=a+c \Rightarrow b=c$, for any elements $a,b,c$. It is easy to check that $\Mon_{lc}$ is not a Mal'tsev category. For example, the ordinary relation $\le$ defined on $\N_0$ is not difunctional:
\[
	\begin{array}{ccc}
	 7 & \le & 8 \\
	 5 & \le & 8 \\
	 5 & \le & 6
	\end{array}
\]
although $7 \nleqslant 6$. However, $\Mon_{lc}$ admits an \Ord-enrichment that makes it an $\Ord$-Mal'tsev category. We shall consider on each $\Mon_{lc}(X,Y)$ the preorder defined by:
\begin{equation}\label{moi for monoids}
	f\moi g \Leftrightarrow \forall x\in X, \exists (!) y_x\in Y : f(x)+ y_x = g(x).
\end{equation}
Note that:
\begin{itemize}
	\item from the left cancellation property such $y_x$ is unique, for each $x\in X$;
	\item the zero morphism $0\colon X\to Y$ is such that $0\moi f$, for any $f\colon X\to Y$.
\end{itemize}
We denote this $\Ord$-category by $\MMon_{lc}$.

Let $\la d_1,d_2\ra \colon D\mono X\times Y$ be an ideal in $\MMon_{lc}$ and suppose we have morphisms $f,h,h'\colon A\to X$ and $g,g',k\colon A \to Y$ such that
\begin{equation}
\label{assumption on D}
	\begin{array}{ccc}
	 f & D & g \vspace{-10pt}\\
	   &   & \rotatebox{-90}{$\moi$} \\
	 h & D & g' \vspace{-10pt}\\
	\rotatebox{-90}{$\moi$} \\
	 h' & D & k \\
	\end{array}
\end{equation}
We want to prove that $f D k$. From $(f,g)\in_A D$ and $0\moi f$, then $(0,g)\in_A D$, since $D$ is an ideal. Since $\la d_1,d_2\ra$ is an ff-(mono)morphism, and given the factorisations
\begin{equation}
\label{<0,g> <= <f,g>}
\vcenter{\xymatrix@=50pt{ & X \\
		A \ar@/_/[ur]_-{f}  \ar@/^/[ur]^-0 \ar@{}[ur]|-{\rotatebox{-40}{$\moi$}} \ar[dr]_-{g} \ar@/^/@{.>}[rr]^-{\alpha} \ar@/_/@{.>}[rr]_-{\beta}
		\ar@{}[rr]|-{\rotatebox{-90}{$\moi$}} & & D \ar[ul]_-{d_1} \ar[dl]^-{d_2}\\
		& Y}}
\end{equation}
(where $\la d_1,d_2\ra \alpha=\la 0,g\ra$ and $\la d_1,d_2\ra \beta = \la f,g\ra$), we conclude that $\alpha \moi \beta$. By \eqref{moi for monoids} this means that, for any element $a$ of $A$, there exists an element $\delta_a$ in $D$ such that $\alpha(a) + \delta_a = \beta(a)$. Let $\delta_a=(\delta_a^1, \delta_a^2)$, for each $a$. We obtain
\[
\left\{\begin{array}{l}
	0 + \delta_a^1 = f(a) \\
	g(a) + \delta_a^2 = g(a).
\end{array}\right.
\]
It follows that $\delta_a^1=f(a)$ and $\delta_a^2=0$ (by left cancellation), for each $a$. Consequently, $(f(a), 0)\in D$, for any element $a$ in $A$. On the other hand, $(h',k)\in_A D$ and $0\moi h'$ gives $(0,k)\in_A D$, since $D$ is an ideal. Then, $(0,k(a))\in D$, for any element $a$ in $A$. Since $D$ is a submonoid of $X\times Y$, we conclude that $(f(a),0) + (0, k(a))=(f(a),k(a))\in D$, for any element $a$ in $A$. This proves that $f D k$, as desired.

We could also consider the preorder by $f\moi g \Leftrightarrow \forall x\in X, \exists (!) y_x\in Y : f(x)= g(x)+y_x.$ In that case $f\moi 0$ for any morphism $f$. Similar arguments show that the category of monoids with right cancellation is an $\Ord$-Mal'tsev category.
\end{example}

\begin{example}\label{Ex:gregarious monoids}
Let $\GMon$ denote the category of gregarious monoids. A monoid $(X,+,0)$ is called \emph{gregarious} when:
\[
	\forall x\in X, \exists u_x,v_x\in X: u_x+x+v_x=0.
\]
Again, we use additive notation although the monoid is not necessarily abelian. We show that $\GMon$ is not a Mal'tsev category. Consider a monoid $M$ generated by two elements $x$ and $y$, which satisfy $x+y=0$. It follows that $M=\{my+nx: m,n\in \N_0\}$. This gives an example of a gregarious monoid which is not a group. It is gregarious since
\[
	\forall\, my+nx\in M, \exists\, mx, ny\in M: mx + (my + nx) +ny=0.
\]
It is not a group since $y+x$, for instance, has no inverse (see~\cite[Example 1.9.4]{BB}). We give an example\footnote{This example is due to Andrea Montoli.} of an ordinary relation $D$ on $M$ which is not difunctional. Consider the submonoid
\[
	D=\{(my+nx,my+nx): m,n\in \N_0\} \cup \{(my+nx,(m+1)y+(n+1)x): m,n\in \N_0\}
\]
of $M\times M$. It is easy to check that it is a gregarious monoid, so that $D\mono M\times M$ is indeed an ordinary relation in $\GMon$. It is not difunctional:
\[
	\begin{array}{ccc}
	 x & D & 2x \\
	 2x & D & 2x \\
	 2x & D & 3x
	\end{array}
\]
although $x D \!\!\!\! / \;\; 3x$.

We consider the same $\Ord$-enrichment of Example~\ref{Ex:Mon with cancellation}; let $\GGMon$ denote this $\Ord$-category. We show that $\GGMon$ is an $\Ord$-Mal'tsev category next. Let $D\colon X\looparrowright Y$ be an ideal in $\mathbb{G}\mathsf{Mon}$ and suppose we have morphisms $f,h,h'\colon A\to X$ and $g,g',k\colon A \to Y$ such that
the relations in \eqref{assumption on D} hold. Since $(f,g)\in_A D$, then $(f(a),g(a))\in D$, for all elements $a$ in $A$. Being gregarious, there exist elements $u_a,v_a\in A$ such that $u_a+a+v_a=0$, for each $a$. As in Example~\ref{Ex:Mon with cancellation}, we also know that $(0,g)\in_A D$. So, each $(0,g(u_a)), (0,g(v_a))\in D$. We deduce that, for all $a\in A$
\[
	(0,g(u_a))+(f(a),g(a))+(0,g(v_a))=(f(a), g(u_a+a+v_a)) = (f(a),0)\in D.
\]
Using arguments similar to those of the final part of Example~\ref{Ex:Mon with cancellation} we conclude that $fDk$.
\end{example}

\begin{example}\label{OrdGrp}
Consider the category $\OrdGrp$ of preordered groups and monotone group homomorphisms. Recall that a preordered group is a (not necessarily abelian) group $(X,+,0)$ equipped with a preorder $\le$ such that the group operation is monotone
\[x\le y, u\le v\;\; \Rightarrow\;\; x+u \le y+v,\]
for any elements $x,y,u,v\in X$; their morphisms are the monotone group homomorphisms.
Note that the preorder of a group $(X,+,0)$ is completely determined by its \emph{positive cone}, which is the submonoid of $X$, closed under conjugation, given by its positive elements, $P_X=\{x\in X: 0\leqslant x\}$.

It was shown in~\cite{CM-FM} that $\OrdGrp$ is not a Mal'tsev category. However, we shall consider an $\Ord$-enrichment for $\OrdGrp$ that makes it an $\Ord$-Mal'tsev category. A similar study was done concerning the protomodularity and suitable $\Ord$-enriched version of protomodularity for ordered (abelian) groups -- see~\cite{CMR}.

In $\OrdGrp$ the pointwise preorder on morphisms trivialises; that is, if one defines, for morphisms $f,g\colon X\to Y$, $f\moi g$ if, for all $x\in X$, $f(x)\leq g(x)$, then also $f(-x)\leq g(-x)$, and consequently, $\moi$ is symmetric. That is, $f\moi g$ only if $f\sim g$. The proof that the $\Ord$-category given by this preorder is not an $\Ord$-Mal'tsev category uses arguments similar to those used to prove that $\OrdGrp$ is not a Mal'tsev category.

We consider now the pointwise order restricted to positive elements, and define, for morphisms $f,g\colon X \to Y$ of $\OrdGrp$,
\begin{equation}
\label{< for OrdAb}
	f\preccurlyeq g \;\;\Leftrightarrow\;\; \forall x\in P_X, f(x)\leq g(x).
\end{equation}
It is straightforward to check that (pre)composition preserves the preorder of $\OrdGrp(X,Y)$, for any preordered groups $X$ and $Y$, and so this defines an $\Ord$-category $\OOrdGrp$. As for the previous examples, we also have $0\moi f$, for any morphism $f$ in $\OOrdGrp$.

Any ideal $D\colon X\looparrowright Y$ in $\OOrdGrp$ is such that $D\cong X\times Y$, as a group. Indeed, consider the ordered group $A=(X\times Y,\{(0,0)\})$ with only one positive element $(0,0)$. The morphisms $0_X\colon X\times Y \to X$ and $0_Y\colon X\times Y \to Y$ are such that $(0_X,0_Y)\in_{A} D$. The product projections $\pi_X\colon X\times Y\to X$ and $\pi_Y\colon X\times Y\to Y$ are such that $\pi_X\moi 0_X$ and $0_Y\moi \pi_Y$. Since $D$ is an ideal, we conclude that $(\pi_X,\pi_Y)\in_{A} D$, i.e. $D\cong X\times Y$. Consequently, the relation $D$ in $\OOrdGrp$ is given by the identity group homomorphism on $X\times Y$, which is also a monotone map
\[
	\xymatrix{(D\cong X\times Y,P_D) \ar[r]^-{1_{X\times Y}} & (X\times Y, P_{X\times Y}=P_X\times P_Y).}
\]

Suppose we have morphisms $f,h,h'\colon (A,P_A) \to (X,P_X)$ and $g,g',k\colon (A,P_A) \to (Y,P_Y)$ such that
the relations in \eqref{assumption on D} hold. We want to prove that $fDk$. There is always a group homomorphism $\la f,k\ra\colon A \to D\cong X\times Y$. To have $fDk$, this group homomorphism must also be a monotone map, i.e. for any positive element $a\in P_A$, we must prove that $(f(a),k(a))\in P_D$.

From diagram \eqref{<0,g> <= <f,g>} of Example~\ref{Ex:Mon with cancellation}, we know that $\la 0,g\ra \moi \la f,g\ra$. This means that, for all positive elements $a\in P_A$, $(0,g(a))\le (f(a),g(a))$; it follows that $(f(a),0)\in P_D$, for each $a\in P_A$. We also know that $(0,k)\in_{(A,P_A)} D$, from which we conclude that $(0,k(a))\in P_D$ for each $a\in P_A$. Since $P_D$ is a submonoid of $D$, we get $(f(a),0)+(0,k(a))=(f(a),k(a))\in P_D$, for each $a\in P_A$.
\end{example}

Note that in the three previous examples one does not use all the assumptions of \eqref{assumption on D}. From the definition of the preorder $\moi$, we deduce a very strong property: $0\moi f$, for any $f\colon X\to Y$. This key property practically solves the issue on its own. 

\begin{remark} The $\Ord$-categories $\MMon_{lc}$, $\GGMon$ and $\OOrdGrp$ are not regular since they do not admit comma objects, thus (R1) fails. The fact that $\OOrdGrp$ does not admit comma objects was shown in~\cite{CMR}. Next we show that $\MMon_{lc}$ does not admit the comma object $1_{\N}/1_{\N}$. Let us suppose that it does exist and is given by the diagram
\[
\xymatrix@=30pt{C \ar[r]^-{\pi_2} \ar[d]_-{\pi_1} \ar@{}[dr]|-{\moi} & \N \ar[d]^-{1_{\N}} \\
								\N \ar[r]_-{1_{\N}} & \N.}
\]
Since $\pi_1\moi \pi_2$, it follows that $\pi_1(c)\le \pi_2(c)$, for all $c\in C$. Consider the morphisms $f,f'$, $g,g'\colon \N \to$ $\N$ defined by $f(n)=n$, $f'(n)=3n$, $g(n)=4n$ and $g'(n)=5n$. We have $f\moi g$, $f'\moi g'$, so that there exist induced morphisms $\lambda=\la f,g\ra\colon \N\to\N, \lambda'=\la f',g'\ra\colon \N\to \N$ from the universal property of the above comma object. From $f\moi f'$ and $g\moi g'$ we deduce that $\lambda\moi \lambda'$ (see Definition~\ref{comma obj def}(C3)). In particular, there exists $c\in C$ such that $\lambda(1)+c=\lambda'(1)$. It then follows that
\[
\left\{\begin{array}{l}
	f(1)+\pi_1(c)=f'(1) \\
	g(1)+\pi_2(c)=g'(1)
\end{array}\right.
\Rightarrow
\left\{\begin{array}{l}
	\pi_1(c)=2 \\
	\pi_2(c)=1.
\end{array}\right.
\]
However, $\pi_1(c)\nleqslant \pi_2(c)$, which contradicts the above assertion that $\pi_1\moi \pi_2$.

Next we show that $\GGMon$ does not admit comma objects. To do so we consider the gregarious monoids $M$ and $D$ from Example~\ref{Ex:gregarious monoids}. Let us suppose that the comma object $1_M/1_M$ exists and is given by the diagram
\[
\xymatrix@=30pt{C \ar[r]^-{\pi_2} \ar[d]_-{\pi_1} \ar@{}[dr]|-{\moi} & M \ar[d]^-{1_{M}} \\
								M \ar[r]_-{1_{M}} & M.}
\]
Consider the morphisms $d_1,d_2\colon D \to M$ defined by first and second projections. It is easy to prove that $d_1\moi d_2$. We have $0\moi d_1$, $d_1\moi d_2$, so that there exist induced morphisms $\lambda=\la 0,d_1\ra\colon \N\to\N, \lambda'=\la d_1,d_2\ra\colon D\to M$ from the universal property of the above comma object. From $0\moi d_1$ and $d_1\moi d_2$ we deduce that $\lambda\moi \lambda'$ (see Definition~\ref{comma obj def}(C3)). In particular, there exists $c\in C$ such that $\lambda(y,y)+c=\lambda'(y,y)$. It then follows that
\[
\left\{\begin{array}{l}
	0+\pi_1(c)=d_1(y,y) \\
	d_1(y,y)+\pi_2(c)=d_2(y,y)
\end{array}\right.
\Rightarrow
\left\{\begin{array}{l}
	\pi_1(c)=y \\
	y+\pi_2(c)=y.
\end{array}\right.
\Rightarrow
\left\{\begin{array}{l}
	\pi_1(c)=y \\
	\pi_2(c)=0
\end{array}\right.
\]
($y+\pi_2(c)=y\Rightarrow x+y+\pi_2(c)=x+y\Rightarrow \pi_2(c)=0$). Since $\pi_1\moi \pi_2$, there exists $z\in M$ such that $\pi_1(c)+z=\pi_2(c)$, i.e. $y+z=0$. Such an equality is impossible to hold due to the definition of $M$. Indeed, $y+z=0\Rightarrow x+y+z=x\Rightarrow z=x$. However, $y+x$ cannot be equal to 0 since it has no inverse, as observed in Example~\ref{Ex:gregarious monoids}.
\end{remark}

\begin{example}\label{internal preorders} Let $\Ccal$ be a regular category. Consider the category of (internal) ordinary preorders in $\Ccal$, which we denote by $\O(\Ccal)$. The objects are pairs $(X,R)$, where $R:X\relto X$ is an ordinary reflexive and transitive relation in $\Ccal$. A morphism $(f,\overline{f})\colon (X,R)\to (Y,S)$ is a pair of morphisms $f\colon X\to Y$ and $\overline{f}\colon R\to S$ of $\Ccal$ such that the following diagram commutes
\[
	\xymatrix{ R \ar[r]^-{\overline{f}} \ar[d]_-{\la r_1,r_2\ra} & S \ar[d]^-{\la s_1,s_2 \ra} \\ X\times X \ar[r]_-{f\times f} & Y\times Y.}
\]
It is easy to check that $(f,\overline{f})$ is a monomorphism in $\O(\Ccal)$ if and only if $f$ and, consequently, $\overline{f}$ are monomorphisms in $\Ccal$. Moreover, $(f,\overline{f})$ is an ff-monomorphism if and only if $f$ is a monomorphism and the commutative diagram above is a pullback in $\Ccal$, i.e. $R\cong f^{-1}(S)$.

We follow~\cite{Vassilis} and consider the preorder on the morphisms in $\O(\Ccal)$ defined as: given $(f,\overline{f})$, $(g,\overline{g})\colon (X,R)\to (Y,S)$, $(f,\overline{f})\moi (g,\overline{g})$ if and only if $(f,g)\in_X S$. We denote this $\Ord$-category by $\OO(\Ccal)$. Following similar arguments as those in~\cite{Vassilis}, we can prove that $\OO(\Ccal)$ satisfies conditions (R1), (R2) and (R3) of Definition~\ref{enriched regular cat}. Consequently, the calculus of relations developed in the previous sections holds for $\OO(\Ccal)$ (see Remark~\ref{remark on Ord-quasiregular}).

\begin{proposition} A regular category $\Ccal$ is a Mal'tsev category if and only if $\OO(\Ccal)$ is an $\Ord$-Mal'tsev category.
\end{proposition}
\begin{proof}
If $\Ccal$ is a regular Mal'tsev category, then any ordinary reflexive relation is necessarily symmetric by Theorem~\ref{from CKP}(vii). So every ordinary preorder $R\colon X\relto X$ in $\Ccal$ is necessarily an ordinary equivalence relation in $\Ccal$. We show that $\O(\Ccal)$ is a Mal'tsev category by using Definition~\ref{Mal'tsev cat with no lims}. Any ordinary relation $(X,R) \stackrel{(\delta_1,\overline{\delta_1})}{\longleftarrow} (Z,D) \stackrel{(\delta_2,\overline{\delta_2})}{\longrightarrow} (Y,S)$ in $\O(\Ccal)$ gives two ordinary relations $R\stackrel{\overline{\delta_1}}{\longleftarrow} D \stackrel{\overline{\delta_2}}{\longrightarrow} S$ and $X\stackrel{\delta_1}{\longleftarrow} Z \stackrel{\delta_2}{\longrightarrow} Y$ in $\Ccal$, which are both difunctional by assumption. It easily follows that  $(X,R) \stackrel{(\delta_1,\overline{\delta_1})}{\longleftarrow} (Z,D) \stackrel{(\delta_2,\overline{\delta_2})}{\longrightarrow} (Y,S)$ is difunctional in $\O(\Ccal)$. By Proposition~\ref{Mal'tsev => 2-Mal'tsev} any $\Ord$-enrichment of $\O(\Ccal)$ is an $\Ord$-Mal'tsev category, thus $\OO(\Ccal)$ is an $\Ord$-Mal'tsev category.

For the converse, consider an ordinary relation $X\stackrel{d_1}{\longleftarrow} D \stackrel{d_2}{\longrightarrow} Y$ in $\Ccal$. The following morphism  $(\la d_1,d_2\ra, \la d_1,d_2\ra)\colon (D,\Delta_D)\to (X\times Y,\Delta_{X\times Y})$
\[
	\xymatrix@C=50pt{D \pullback\ar[d]_-{\la 1_D,1_D\ra} \ar[r]^-{\la d_1,d_2\ra} & X\times Y \ar[d]^-{\la 1_{X\times Y}, 1_{X\times Y} \ra} \\ D\times D \ar[r]_-{\la d_1,d_2\ra\times \la d_1,d_2\ra} & X\times Y\times X\times Y}
\]
is an ff-monomorphism in $\OO(\Ccal)$ since $\la d_1,d_2\ra$ is a monomorphism in $\Ccal$ and the diagram above is a pullback in $\Ccal$. The relation $(X,\Delta_X) \stackrel{(d_1,d_1)}{\longleftarrow} (D,\Delta_D) \stackrel{(d_2,d_2)}{\longrightarrow} (Y,\Delta_Y)$ is an ideal in $\OO(\Ccal)$. Indeed, suppose that $(a,\overline{a}),(x,\overline{x})\colon (A,U)\to (X,\Delta_X)$ and $(y,\overline{y}),(b,\overline{b})\colon (A,U)\to (Y,\Delta_Y)$ are such that $((x,\overline{x}),(y,\overline{y}))\in_{(A,U)} (D,\Delta_D)$, $(a,\overline{a})\moi (x,\overline{x})$ and $(y,\overline{y})\moi (b,\overline{b})$. Then $(a,x)\in_A \Delta_X$ and $(y,b)\in_A \Delta_Y$, from which we conclude that $a=x$ and $b=y$; consequently, $(a,\overline{a})=(x,\overline{x})$ and $(b,\overline{b})=(y,\overline{y})$.

Suppose that $x,u\colon A\to X$ and $y,v\colon A\to Y$ are such that $(x,y)\in_A D$, $(u,y)\in_A D$ and $(u,v)\in_A D$ (see \eqref{picture of difunctional}). It follows that $((x,x),(y,y))\in_{(A,\Delta_A)} (D,\Delta_D)$, $((u,u),(y,y))\in_{(A,\Delta_A)} (D,\Delta_D)$ and $((u,u),(v,v))\in_{(A,\Delta_A)} (D,\Delta_D)$. By assumption the ideal $(D,\Delta_D)\colon (X,\Delta_X)\looparrowright (Y,\Delta_Y)$ is difunctional, so that $((x,x),(v,v))\in_{(A,\Delta_A)} (D,\Delta_D)$. In particular, we get $(x,v)\in_A D$.
\end{proof}

\end{example}

\begin{example}\label{ex:weakly}
We recall from \cite[Corollary 5.1]{JMF} (see also \cite{NMFweakly}) that a category with pullbacks and equalisers is \emph{weakly Mal'tsev} if and only if every strong ordinary relation is difunctional. By strong ordinary relation we mean an ordinary relation $X\stackrel{r_1}{\longleftarrow} R \stackrel{r_2}{\longrightarrow} Y$ such that $(r_1,r_2)$ is jointly strongly monomorphic. Hence a weakly Mal'tsev category with an $\Ord$-enrichment such that every ff-monomorphism is strong is automatically an $\Ord$-Mal'tsev category.
\end{example}

\begin{example}\label{ex:VCat}
As shown in \cite[Proposition 3]{NMFnew}, the category $(\VCat)^\op$, for a fixed unital and integral quantale $V=(V,\moi,\otimes,k)$, is weakly Mal'tsev. It is also a quasivariety (see \cite{Zhang, CFH}), hence in particular it is a regular category. Moreover, we have shown in \cite{ClementinoRodelo} that the full subcategory of symmetric $V_\wedge$-categories is a Mal'tsev category. The category $(\VCat)^\op$ has a natural $\Ord$-enrichment given, for every $V$-functor $f\colon X\to Y$, by $f\moi g$ if, for all $x\in X$, $Y(f(x),g(x))=k$. It is easy to check that ff-monomorphisms $f\colon X\to Y$ in $(\VCat)^\op$ are exactly surjective $V$-functors $Y\to X$, while a strong monomorphism is a final surjection, so that $X(x,x')=\bigvee\left\{Y(y,y'); y\in f^{-1}(x), y'\in f^{-1}(x')\right\}$. Therefore ff-monomorphisms do not need to be strong, and so one cannot conclude that $(\VCat)^\op$ is an $\Ord$-Mal'tsev category. Indeed, using the results of the Appendix we show next that a $V$-category is an $\Ord$-W-Mal'tsev object in $(\VCat)^\op$ if and only if it is a symmetric $V_\wedge$-category, showing this way that $(\VCat)^\op$ is not an $\Ord$-Mal'tsev category.

First of all we should note that $(\VCat)^\op$ is a regular $\Ord$-category. Here, in order to calculate $R_*$ for a given relation $R$, we need to build the cocomma object of the identities on an object $X$:
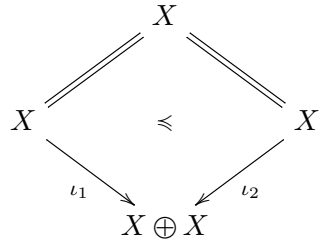
\[\xymatrix{&X\ar@{=}[ld]\ar@{=}[rd]\\
X\ar[rd]_{\iota_1}\ar@{}[rr]|{\moi}&&X\ar[ld]^{\iota_2}\\
&X\oplus X&}\]
It is straightforward to check that $X\oplus X$ has as underlying set $X+X=X\times\{1\}\cup X\times\{2\}$, where $X\oplus X((x,i),(x',j))$ is either $X(x,x')$ if $i\leq j$, or $\bot$ (if $i=2$ and $j=1$). Hence, given a relation $X\stackrel{r_1}{\longleftarrow} R \stackrel{r_2}{\longrightarrow} Z$, $R_*$ is defined by the following diagram
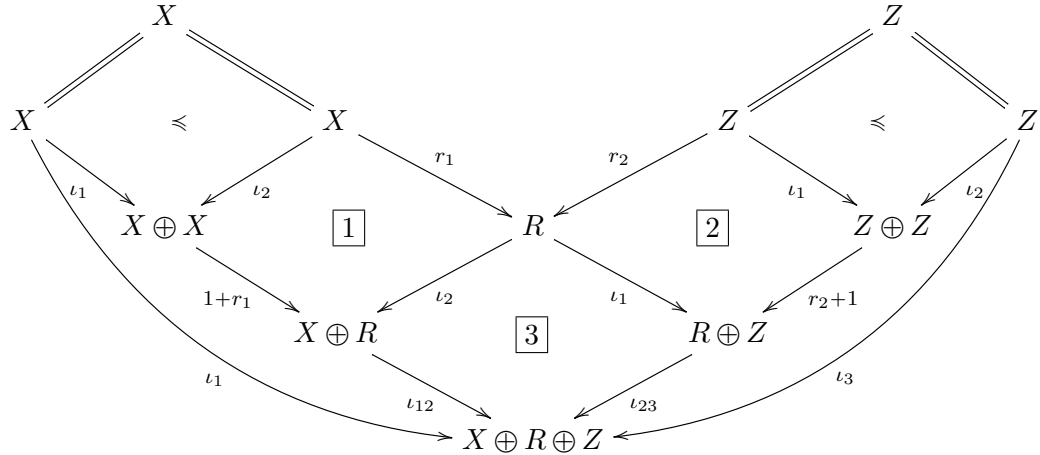
\[\xymatrix{&X\ar@{=}[ld]\ar@{=}[rd]&&&&Z\ar@{=}[ld]\ar@{=}[rd]\\
X\ar@/_3pc/[rrrddd]_{\iota_1}\ar[rd]_{\iota_1}\ar@{}[rr]|{\moi}&&X\ar[ld]^{\iota_2}\ar[rd]^{r_1}&&Z
\ar[ld]_{r_2}\ar[rd]_{\iota_1}\ar@{}[rr]|{\moi}&&Z\ar@/^3pc/[lllddd]^{\iota_3}\ar[ld]^{\iota_2}\\
&X\oplus X\ar@{}[rr]|{\fbox{1}}\ar[rd]_{1+r_1}&&R\ar@{}[rr]|{\fbox{2}}\ar[ld]^{\iota_2}\ar[rd]_{\iota_1}&&Z\oplus Z\ar[ld]^{r_2+1}\\
&&X\oplus R\ar@{}[rr]|{\fbox{3}}\ar[rd]_{\iota_{12}}&&R\oplus Z\ar[ld]^{\iota_{23}}\\
&&&X\oplus R\oplus Z}\]
where \fbox{1}, \fbox{2}, \fbox{3} are pushouts.
That is, the underlying set of $X\oplus R$ is $X+R$, with
\[(X\oplus R)(w,w')=\left\{\begin{array}{ll}
X(w,w')&\mbox{ if }w,w'\in X\\
R(w,w')&\mbox{ if }w,w'\in R\\
R(r_1(w),w')&\mbox{ if }w\in X,\,w'\in R\\
\bot&\mbox{ if }w\in R,w'\in X.
\end{array}\right.\]
Likewise for $R\oplus Z$ and $X\oplus R\oplus Z$. Now $R_*$ is obtained via the (epimorphism, strong mo\-no\-mor\-phism)-factorisation of the morphism $\left(\begin{smallmatrix} \iota_1 \\ \iota_3\end{smallmatrix}\right)\colon X+Z\to X\oplus R\oplus Z$. Therefore the underlying set of $R_*$ is $X+Z$, with
$R_*(w,w')=(X\oplus R\oplus Z)(w,w')$; that is,
\[R_*(w,w')=\left\{\begin{array}{ll}
X(w,w')&\mbox{ if }w,w'\in X\\
Z(w,w')&\mbox{ if }w,w'\in Z\\
R(r_1(w),r_2(w'))&\mbox{ if }w\in X,\, w'\in Z\\
\bot&\mbox{ if }w\in Z,\, w'\in X.
\end{array}\right.\]

\begin{proposition}
A $V$-category is an $\Ord$-W-Mal'tsev object in $(\VCat)^\op$ if and only if it is a symmetric $V_\wedge$-category.
\end{proposition}
\begin{proof}
To proof the claim we will make use of Proposition \ref{handy wrt Ord-enriched W-Mal'tsev objs}. Let $Y$ be a $V$-category and $D$ be the relation defined in \eqref{iotas}. Our aim is to check under which conditions the map $h\colon D_*\to Y$ making the following diagram commute, so that $h=\left(\begin{smallmatrix} \pi_1 \\ \pi_1\end{smallmatrix}\right)$, is a $V$-functor:
\[\xymatrix{Y\times Y\times Y\\
&D\ar[lu]_m&D_*\ar[l]_{e'}\ar[rd]^h\\
(Y\times Y)+(Y\times Y)\ar[uu]^{\left(\begin{smallmatrix} \pi_1 & \pi_2 & \pi_2 \\ \pi_2 & \pi_2 & \pi_1 \end{smallmatrix}\right)}\ar[ru]^e\ar[rru]_{\id}\ar[rrr]_{\left(\begin{smallmatrix} \pi_1 \\ \pi_1\end{smallmatrix}\right)}&&&Y}\]
Here $e'(y_1,y_2)=e(y_1,y_2)=(y_1,y_2,y_2)$ if $(y_1,y_2)$ belongs to the first summand, and \linebreak $e'(y_1,y_2)=e(y_1,y_2)=(y_2,y_2,y_1)$ if $(y_1,y_2)$ belongs to the second one. Then $h$ is a $V$-functor if and only if, for all $(y_1,y_2), (y_1',y_2')$ in $(Y\times Y)+(Y\times Y)$,
\begin{equation}\label{eq:Vfunctor}
D_*((y_1,y_2), (y_1',y_2'))\leq Y(y_1,y_1').
\end{equation}
When $(y_1,y_2)$ belongs to the first summand and $(y_1',y_2')$ belongs to the second one this means, since $D_*((y_1,y_2),(y_1',y_2'))=D((y_1,y_2,y_2),(y_2',y_2',y_1'))$,
\begin{equation}\label{eq:D*}
Y(y_1,y_2')\wedge Y(y_2,y_2')\wedge Y(y_2,y_1')\leq Y(y_1,y_1').
\end{equation}
Taking $y_1=y_2'$ this inequality translates to
\[Y(y_2,y_1)\wedge Y(y_2,y_1')\leq Y(y_1,y_1'),\]
which is equivalent to $Y$ being a symmetric $V_\wedge$-category (see \cite[Theorem 2.4]{ClementinoRodelo}).

Conversely, to show that $h$ is a $V$-functor provided that $Y$ is a symmetric $V_\wedge$-category, we note that the inequality \eqref{eq:Vfunctor} is trivially satisfied in all cases but when $(y_1,y_2)$ belongs to the first summand and $(y_1',y_2')$ belongs to the second one. In this case we have to show that \eqref{eq:D*} holds, for all $y_1,y_2,y_1',y_2'\in Y$: using first symmetry and then transitivity of the $V_\wedge$-category $Y$ we obtain:
\[Y(y_1,y_2')\wedge Y(y_2,y_2')\wedge Y(y_2,y_1')=Y(y_1,y_2')\wedge Y(y_2',y_2)\wedge Y(y_2,y_1')\leq Y(y_1,y_1').\]
\end{proof}
\end{example}

\appendix
\section{An object-wise approach to \texorpdfstring{$\Ord$}{Ord}-Mal'tsev categories}\label{An object-wise approach to Ord-enriched Mal'tsev categories}

The authors of~\cite{MRVdL} explored several algebraic categorical notions at an object-wise level. One of those was the notion of  Mal'tsev object. Their approach was inspired by the classification properties of the fibration of points studied in~\cite{MCFPO}. Independently, the author of~\cite{Weighill} used the characterisation of a Mal'tsev category obtained through the difunctionality of ordinary relations to introduce a definition of Mal'tsev object (recall Definition~\ref{Mal'tsev cat with no lims} and Theorem~\ref{from CKP}). A comparison between both notions may be found in~\cite{ClementinoRodelo}, where a Mal'tsev object in the sense of~\cite{Weighill} was called a ``W-Mal'tsev object''; we keep that designation in this work.

\begin{definition}[\cite{Weighill}]\label{W-Maltsev obj}
An object $Y$ of a category $\Ccal$ is called a \defn{W-Mal'tsev object} when for every ordinary relation $X\stackrel{r_1}{\longleftarrow} R \stackrel{r_2}{\longrightarrow} Z$ in $\Ccal$, the $\Set$-relation
\[
	\xymatrix@C=40pt{\Ccal(Y,X) & \ar[l]_-{\Ccal(Y,r_1)} \Ccal(Y,R) \ar[r]^-{\Ccal(Y,r_2)} & \Ccal(Y,Z)}
\]
is difunctional.
\end{definition}

It follows from Definitions~\ref{Mal'tsev cat with no lims} and~\ref{W-Maltsev obj} that a category $\Ccal$ is a Mal'tsev category if and only if all of its objects are W-Mal'tsev objects.

This definition does not impose any kind of assumption on the base category $\Ccal$. However, if $\Ccal$ is regular and admits binary coproducts, the definition of a W-Mal'tsev object becomes easier to handle. Indeed, it allows the replacement of a property on all ordinary relations by a property on a specific ordinary relation defined on coproducts. As usual, for an object $Y$, we write $Y+Y=2Y$, $Y+Y+Y=3Y$, and $\iota_j\colon Y\to kY$ for the $j$-th coproduct coprojection.

\begin{proposition}[\cite{Weighill}]\label{handy wrt W-Mal'tsev objs} Let $\Ccal$ be a regular category with binary coproducts. An object $Y$ is a W-Mal'tsev object in $\Ccal$ if and only if,
	given the (regular epimorphism, monomorphism) factorisation in $\Ccal$
\[
\vcenter{\xymatrix{ 3Y \ar[dr]^e \ar[dd]_-{\left(\begin{smallmatrix} \iota_1 & \iota_2 \\ \iota_2 & \iota_2 \\ \iota_2 & \iota_1 \end{smallmatrix}\right)} \\ & D \ar[dl]^-{\la d_1,d_2\ra} \\ 2Y\times 2Y}}
\]
(which guarantees that $(\iota_1,\iota_2)\in_Y D$, $(\iota_2,\iota_2)\in_Y D$, $(\iota_2,\iota_1)\in_Y D$), we have $(\iota_1,\iota_1)\in_Y D$.
\end{proposition}

We present the $\Ord$-enriched versions of this approach next.

\begin{definition}\label{Ord-enriched W-Maltsev obj}
An object $Y$ of an $\Ord$-category $\C$ is called an \defn{$\Ord$-W-Mal'tsev object} when every ideal $R\colon X\looparrowright Z$ satisfies
\[
	\begin{array}{ccc}
	 x & R & z \vspace{-10pt}\\
	   &   & \rotatebox{-90}{$\moi$} \\
	 u & R & z' \vspace{-10pt}\\
	\rotatebox{-90}{$\moi$} \\
	 u' & R & v \\
	\hline
	x & R & v
	\end{array}
\]
for any generalised elements $x, u, u'\colon Y\to X$, $z,z',v\colon Y\to Z$.
\end{definition}

It follows from Definitions~\ref{Ord-Mal'tsev} and~\ref{Ord-enriched W-Maltsev obj} that \emph{an $\Ord$-category $\C$ is an $\Ord$-Mal'tsev category if and only if all of its objects are $\Ord$-W-Mal'tsev objects}.

\begin{proposition}\label{handy wrt Ord-enriched W-Mal'tsev objs} Let $\C$ be a regular $\Ord$-category with binary coproducts. An object $Y$ is an $\Ord$-W-Mal'tsev object in $\C$ if and only if,
	given the (so-morphism, ff-monomorphism) factorisation in $\C$
\begin{equation}\label{iotas}
\vcenter{\xymatrix{ 3Y \ar@{>>}[dr]^e \ar[dd]_-{\left(\begin{smallmatrix} \iota_1 & \iota_2 \\ \iota_2 & \iota_2 \\ \iota_2 & \iota_1 \end{smallmatrix}\right)} \\ & D \ar@{ >->}[dl]^-{\la d_1,d_2\ra} \\ 2Y\times 2Y}}
\end{equation}
(which guarantees that $(\iota_1,\iota_2)\in_Y D$, $(\iota_2,\iota_2)\in_Y D$, $(\iota_2,\iota_1)\in_Y D$), we have $(\iota_1,\iota_1)\in_Y D_*$.
\end{proposition}
\begin{proof} Suppose that $Y$ is an $\Ord$-W-Mal'tsev object. Diagram \eqref{iotas} tells us that $\iota_1 D \iota_2$, $\iota_2 D \iota_2$ and $\iota_2 D \iota_1$. We also know that $D\subseteq D_*$ by Lemma~\ref{R^* smallest ideal contains R^op}. We get
\[
\begin{array}{ccc}
	 \iota_1 & D_* & \iota_2 \vspace{-10pt}\\
	   &   & \rotatebox{-90}{$\moi$} \\
	 \iota_2 & D_* & \iota_2 \vspace{-10pt}\\
	\rotatebox{-90}{$\moi$} \\
	 \iota_2 & D_* & \iota_1 \\
	\hline
	\iota_1 & D_* & \iota_1
	\end{array}
\]
For the converse, consider an ideal $R\colon X\looparrowright Z$ and the relations as in Definition~\ref{Ord-enriched W-Maltsev obj}. Since $R$ is an ideal, we get $(x,z')\in_Y R$, and $(u,v)\in_Y R$; we have induced morphisms
\[
	\xymatrix@R=10pt{ & X \\
		Y \ar[ur]^-{x} \ar[dr]_-{z'} \ar@{.>}[rr]^-{\exists \alpha} & & R, \ar[ul]_-{r_1} \ar[dl]^-{r_2}\\
		& Z}\;\;\;\;
	\xymatrix@R=10pt{ & X \\
		Y \ar[ur]^-{u} \ar[dr]_-{z'} \ar@{.>}[rr]^-{\exists \beta} & & R, \ar[ul]_-{r_1} \ar[dl]^-{r_2}\\
		& Z}\;\;\;\;
	\xymatrix@R=10pt{ & X \\
		Y \ar[ur]^-{u} \ar[dr]_-{v} \ar@{.>}[rr]^-{\exists \gamma} & & R \ar[ul]_-{r_1} \ar[dl]^-{r_2}\\
		& Z}
\]

Now, consider the 2-pullback and the induced morphism in
\[
	\xymatrix@C=40pt{3Y \ar@/^15pt/[drr]^-{\left(\begin{smallmatrix} \alpha \\ \beta \\ \gamma\end{smallmatrix}\right)} \ar@/_15pt/[ddr]_-{\left(\begin{smallmatrix} \iota_1 & \iota_2 \\ \iota_2 & \iota_2 \\ \iota_2 & \iota_1 \end{smallmatrix}\right)} \ar@{.>}[dr]^-{\sigma} \\
	 & S \pullback \ar[r]^-{\zeta} \ar@{ >->}[d]_-{\la s_1,s_2\ra} & R \ar@{ >->}[d]^-{\la r_1,r_2\ra} \\
	& 2Y\times 2Y \ar[r]_-{\left(\begin{smallmatrix} x\\u\end{smallmatrix}\right)\times \left(\begin{smallmatrix} v\\z'\end{smallmatrix}\right)} & X\times Z.}
\]
From the (so-morphism,ff-monomorphism) factorisation \eqref{iotas}, it follows that $D\subseteq S$. We get $D_*\subseteq S$; let $i$ be the inclusion morphism $i\colon D_* \to S$. By assumption, we have $(\iota_1,\iota_1)\in_Y D_*$, meaning that there exists a factorisation
\[
\xymatrix@R=10pt{ & 2Y \\
		Y \ar[ur]^-{\iota_1} \ar[dr]_-{\iota_1} \ar@{.>}[rr]^-{\exists \tau} & & D_*. \ar[ul] \ar[dl]\\
		& 2Y}
\]
Finally, we get
\[
\xymatrix@R=10pt{ & X \\
		Y \ar[ur]^-{x} \ar[dr]_-{v} \ar@{.>}[rr]^-{\zeta i \tau} & & R, \ar[ul]_-{r_1} \ar[dl]^-{r_2}\\
		& Z}
\]
which proves that $(x,v)\in_Y R$.
\end{proof}

\section*{Acknowledgements}
The authors are grateful to the anonymous referee, who gave suggestions
which helped improving the paper, and to Fernando Lucatelli Nunes for
fruitful discussions on the notion of Ord-regular category.

\end{document}